\documentclass[preprint,5p,times,twocolumn]{elsarticle}

\usepackage{amssymb}
\usepackage{amsmath}
\usepackage{amsfonts,amsthm}
\usepackage{algorithmic,algorithm}
\usepackage{graphicx}
\usepackage{textcomp}
\usepackage{indentfirst}
\usepackage{mathrsfs}
\usepackage{mathdots,color}

\newcommand{\bm}[1]{{\mbox{\boldmath $#1$}}}

\newtheorem{definition}{Definition}
\newtheorem{assumption}{Assumption}

\newtheorem{theo}{Theorem}
\newtheorem{rem}{Remark}

\newtheorem{lem}{Lemma}

\def\rank{{\rm rank} \hspace{0.5mm}}

\def\im{{\rm im} \hspace{0.3mm}}

\def\tr{{\rm tr} \hspace{0.5mm}}
\def\rms{{}^{\rm s}\hspace{-0.2mm}}
\def\rmsK{{}^{\rm s}\hspace{-0.4mm}}
\def\rmsJ{{}^{\rm s}\hspace{-0.5mm}}
\def\h{\hspace{-0.5mm}}
\def\H{\hspace{-1.0mm}}
\def\hH{\hspace{-1.5mm}}
\def\HH{\hspace{-2.0mm}}
\def\HHH{\hspace{-3.0mm}}
\def\qed{\hfill $\blacksquare$}

\journal{Systems \& Control Letters}

\begin{document}

\begin{frontmatter}

\title{Policy Gradient Method for LQG Control via Input–Output‑History Representation: Convergence to $\mathcal{O}(\epsilon)$‑Stationary Points
}


\author[uec]{Tomonori Sadamoto\corref{cor1}}
\ead{sadamoto@uec.ac.jp}

\author[purdue]{Takashi Tanaka}
\ead{tanaka16@purdue.edu}

\cortext[cor1]{Corresponding author}

\affiliation[uec]{
  organization={Department of Mechanical and Intelligent Systems Engineering, 
                Graduate School of Informatics and Engineering, 
                The University of Electro-Communications},
  addressline={1-5-1 Chofugaoka},
  city={Chofu},
  postcode={182-8585},
  state={Tokyo},
  country={Japan}
}

\affiliation[purdue]{
  organization={School of Aeronautics and Astronautics, and of Electrical and Computer Engineering, Purdue University},
  addressline={701W. Stadium Ave.},
  city={West Lafayette},
  postcode={47907‑2045},
  state={Indiana},
  country={USA}
}

\begin{abstract}
We study the policy gradient method (PGM) for the linear quadratic Gaussian (LQG) dynamic output‑feedback control problem using an \emph{input–output‑history} (IOH) representation of the closed‑loop system. First, we show that any dynamic output‑feedback controller is equivalent to a static partial‑state feedback gain for a new system representation characterized by a finite-length IOH. Leveraging this equivalence, we reformulate the search for an optimal dynamic output feedback controller as an optimization problem over the corresponding partial-state feedback gain. 
Next, we introduce a relaxed version of the IOH-based LQG problem by incorporating a small process noise with covariance $\epsilon I$ into the new system to ensure coerciveness, a key condition for establishing gradient‑based convergence guarantees. Consequently, we show that a vanilla PGM for the relaxed problem converges to an {\it $\mathcal{O}(\epsilon)$‑stationary} point, i.e., $\overline{K}$ satisfying $\|\nabla J(\overline{K})\|_F \le \mathcal{O}(\epsilon)$, where $J$ denotes the original LQG cost. Numerical experiments empirically indicate convergence to the vicinity of the globally optimal LQG controller.
\end{abstract}

%




\end{frontmatter}


\section{Introduction}
The linear quadratic Gaussian (LQG) control problem is one of the most fundamental problems in modern control theory \cite{kalman1961new,athans1971role}. Given the explicit knowledge of the system model and the cost function, it is widely recognized that the optimal policy is given by a dynamic output feedback law and can be synthesized by solving two Riccati equations \cite{soderstrom2002discrete}.

Motivated by many practical situations where the system model is unknown, there has been increased research interest in model-free approaches where the policy is learned from the system's input-output behavior. 
To achieve this goal, so-called \emph{indirect} approaches first attempt to estimate the system model \cite{zheng2021sample2, boczar2018finite} from the input-output data, and then apply a classical model-based synthesis.
Alternatively,  \emph{direct} approaches attempt to optimize the policy incrementally using data by certain update rules such as gradient methods.
While the direct approaches are natural given the popularity of the policy gradient methods (PGMs) in the learning theory community, recent studies have revealed prominent challenges in the PGM-based LQG synthesis. 
These challenges are primarily due to the complex optimization landscape underlying the LQG control problems. 
For example, recent works \cite{tang2021analysis,duan2022optimization,hu2023toward} have shown that optimization of the system matrices of dynamic output-feedback controllers has many stationary points. To make matters worse, the optimization problem is {\it not coercive} \cite{hu2023toward}; therefore, it is not even guaranteed that PGMs will converge to one of these numerous stationary points. Consequently, most studies have been limited to analyzing the optimization landscape, while no theoretical analysis has been conducted on the convergence of PGMs.

In this study, we consider a particular parameterization of dynamic controllers and propose a vanilla PGM with a theoretical guarantee of convergence to {\it $\mathcal{O}(\epsilon)$‑stationary} point, i.e., $\overline{K}$ satisfying $\|\nabla J(\overline{K})\|_F \le \mathcal{O}(\epsilon)$, where $J$ denotes the original LQG cost and $\epsilon \ll 1$ is a design parameter. While our analysis assumes that the plant model is known in advance, the fact that the PGM converges to $\mathcal{O}(\epsilon)$‑stationary points offers valuable insights into the convergence analysis of model-free PGM-based LQG synthesis. The contributions of this paper are as follows:

First, we show that designing a dynamic output feedback controller without a feed-through term for a partially observable system affected by process and measurement noise is equivalent to designing a static partial-state feedback gain (which we refer to as the {\it IOH gain} hereafter) for a new system whose internal state consists of a finite-length history of the input-output data and noise. Hereafter, we refer to the new system as {\it history dynamics}. We show how a dynamic output feedback controller can be written as an equivalent IOH gain and \emph{vice versa}. This equivalence enables us to search for $\mathcal O(\epsilon)$-stationary points of the LQG problem via the search for the IOH gain. 

Second, we propose a vanilla PGM with a theoretical guarantee of convergence to $\mathcal O(\epsilon)$-stationary points of the LQG problem. The LQG cost as a function of the IOH gain is still not coercive in general. To address this issue, we introduce a relaxed version of the LQG problem, where a small process noise with covariance $\epsilon I$ is added to the history dynamics. This relaxation ensures that the cost function becomes coercive over the set of stabilizing IOH gains and smooth over a sublevel set. Furthermore, we show that $\overline{K}$, a stationary point of the relaxed LQG cost, satisfies $\|\nabla J(\overline{K})\|_F \leq \mathcal{O}(\epsilon)$ where $J$ is the original LQG cost function. As a result, we establish that a vanilla PGM for the relaxed problem can find an $\mathcal{O}(\epsilon)$-stationary point of the original LQG cost. While this theoretical guarantee is limited to the convergence to $\mathcal O(\epsilon)$-stationary points, our numerical study suggests that the PGM converges to the {\it optimal LQG} controller if $\epsilon$ is sufficiently small. Moreover, even when the controller dimension is restricted to be less than that of the system, we show that the PGM can find a better reduced-order controller than a conventional controller reduction method.

\subsection*{Related Works}
Convergence of the PGM for discrete-time linear quadratic regulator (LQR) problems was studied by \cite{fazel}. Despite the non-convex nature of the problem, they proved that the PGM converges to the globally optimal solution because the problem has benign properties such as smoothness, coerciveness, and gradient dominance. This result has been extended to PGMs for continuous-time systems \cite{mohammadi2019global, Mih3}, zero-sum games \cite{bu2019global}, and distributed controls \cite{li2021distributed}. While these results are limited to state feedback synthesis, extensions to static output-feedback synthesis have been considered \cite{fatkhullin2021optimizing,takakura2022structured,bu2019topological,duan2023optimization}. Unlike the state-feedback case, the set of stabilizing controllers is typically disconnected and has many stationary points, making the convergence analysis more complex. Furthermore, even if a globally optimal solution is found, the achievable control performance is inherently more limited than that of {\it dynamic} output-feedback controllers. These challenges have motivated the line of research on dynamic output-feedback synthesis.

While the classical model-based methods for the LQG dynamic output feedback synthesis (e.g., via Riccati equations) are well-established, PGM-based approaches present notable challenges. 
The existing approaches can be organized into three categories in terms of the parameterization methods of dynamic controllers. 
The first category uses the state space model (the system matrices) to parametrize dynamic controllers \cite{tang2021analysis, duan2022optimization, hu2023toward, zheng2023benign}. 
While this approach is naturally inspired by the classical model-based synthesis, its optimization landscapes have turned out to be extremely complex, in part due to the inherent symmetry caused by similarity transformations. This approach enjoys neither gradient dominance nor coerciveness; thus, convergence to at least stationary points is not guaranteed. The second approach is based on a convex reparameterization of the controller \cite{mohammadi2019global,sun2021learning}, which is also well-known within the framework of modern control theory. 
This approach allows for a reformulation of the original LQG control problem as a convex optimization problem in a reparametrized space, which facilitates the search for a global solution. Despite its elegance, convex reparameterization is model-based in that it requires explicit knowledge of the system model. Therefore, this approach is inconvenient in model-free scenarios in which the model is not known.

This paper focuses on the third approach, IOH parameterization, as explored in \cite{sadamoto2024policy, al2022behavioral}. 
The IOH-based approach differs from other parameterizations primarily in a conceptual sense: it offers a compact theoretical lens that recasts dynamic output-feedback controller design as static partial-state feedback controller design on a measurable IOH state, and it provides a bridge for extending results from state-feedback problems to the LQG setting. We note that the PGM over the IOH gains for the LQG control design was proposed in \cite{al2022behavioral}. However, no theoretical analysis was provided regarding the convergence of the algorithm. Dynamic output-feedback controller design via IOH parameterization has been considered for {\it noise-free} systems \cite{sadamoto2024policy}. However, proving the global convergence of PGMs to LQG control problems with noise inputs by extending the result in \cite{sadamoto2024policy} is not straightforward (as discussed in Section~\ref{Sec3C}) and remains an open challenge. In this paper, as a first step towards the convergence analysis, we show that the PGM via the IOH representation converges to $\mathcal O(\epsilon)$-stationary points. To the best of our knowledge, this is the first result to provide a PGM for general LQG control synthesis with a theoretical guarantee of convergence.


\subsection*{Notation}
For a sequence $\{x(t)\}_{t\in\mathbb{Z}}$ of column vectors and two integers $t_1$ and $t_2$ such that $t_1<t_2$,  we write $[x]_{t_1}^{t_2} := [x^{\top}(t_1), \ldots, x^{\top}(t_2)]^{\top}$. 
The Moore–Penrose inverse of a matrix $A \in \mathbb{R}^{n \times m}$ is denoted as $ A^{\dagger}$. 
For $A \in \mathbb R^{n \times n}$, its spectral radius is denoted as $\rho(A)$, the 2-induced norm as $\|A\|$, the Frobenius norm as $\|A\|_F$. 
We write $A > 0$ (resp. $A \geq 0$) if the matrix $A$ is positive definite (resp. positive semidefinite).
We denote by $A^{\frac{1}{2}}\geq 0$ the unique symmetric positive semidefinite square root of the matrix $A \geq 0$.
Given $A \in \mathbb R^{n_x \times n_x}$, $B \in \mathbb R^{n_x \times n_u}$, $C \in \mathbb R^{n_y \times n_x}$, and $L \in \mathbb{N}$, we define
$\mathcal R_L(A,B) := [A^{L-1}B, \ldots, B]$,  
$\mathcal O_L(A,C) := [C^{\top}, \ldots, (CA^{L-1})^{\top}]^{\top}$, and 
$\mathcal H_L(A,B,C) := [H_{i,j}]$ where $H_{i,j} \in \mathbb{R}^{n_y \times n_u}$ is the $(i,j)$-th block matrix defined as $H_{i,j} = 0$ if $i \leq j$, and $H_{i,j} = CA^{i-j-1}B$ otherwise. 

\section{LQG Problem}\label{sec2}
We consider a discrete-time linear system described as 
\begin{equation}\label{hat Sigma}
\rms{\bm \Sigma}: 
\left\{\H
\begin{array}{rcl}
    x(t+1) \HH &=& \HH Ax(t)+B u(t) + w(t)\\
    y(t) \HH &=& \HH Cx(t) +v(t)
\end{array}
\right.\HH , \quad t \geq 0,
\end{equation}
where $x(t) \in {\mathbb R}^{n_x}$ is the state with the initial condition $x(0) = 0$, $u(t) \in \mathbb R^{n_u}$ is the control input determined by the controller, 
$y(t) \in {\mathbb R}^{n_y}$ is the output observed by the controller, $w(t) \in \mathbb R^{n_w}$ is the process noise, and $v(t) \in \mathbb R^{n_v}$ is the measured noise. From \eqref{hat Sigma}, it is clear that $n_w = n_x$ and $n_v = n_y$. However, we retain all symbols for ease of presentation. We assume that signals $x(t)$, $w(t)$, and $v(t)$ cannot be directly measured by the controller. The system matrices $A$, $B$, and $C$ are assumed to be known in advance. In addition, we make the following assumptions. 
\begin{assumption}\label{ass_ABC}
The pair $(A,B)$ is stabilizable and the pair $(A,C)$ is observable. 
\end{assumption}
\begin{assumption}\label{ass_wv}
The noise processes satisfy
\begin{align}
    d(t) \h := \h \left[\hspace{-1.5mm}\begin{array}{c} w(t) \\ v(t) \end{array}\hspace{-1.5mm}\right]
    \hspace{-1mm} \overset{i.i.d.}{\sim}\hspace{-1mm} \mathcal N(0, V_d), \quad V_d \hspace{-0.5mm}=\hspace{-0.5mm} \left[\hspace{-1.5mm}\begin{array}{cc}V_w &  \\  & V_v \end{array}\hspace{-1.5mm}\right] \h \geq  0\label{def_d}
\end{align}
and $V_v > 0$. 
\end{assumption}
In this study, we aim to design a dynamic output-feedback controller of the form
\begin{equation}\label{dyn_K}
    \rmsK{\bm K} : \hspace{0pt}
\left\{\H
\begin{array}{rcl}
    \xi(t+1) \HH &=& \HH G \xi(t) + H y(t)\\
    u(t) \HH &=& \HH F \xi(t)
\end{array}
\right.\HH ,~ \xi \in \mathbb R^{n_{\xi}},~t \geq 0,
\end{equation}
to minimize the cost function
\begin{equation}\label{defJ_orig}
    \hH \rmsJ J(\rmsK {\bm K}) \h := \h \lim_{T \rightarrow \infty}\h \mathbb E \h
    \left[\h \frac{1}{T}\h \sum_{t = 0}^{T} y^{\top}(t)Qy(t) + u^{\top}(t)Ru(t)
    \right]\H ,
\end{equation}
where $Q > 0$ and $R > 0$ are given matrices. The positive definiteness of $Q$ is required to guarantee the stability of the closed-loop system based on the finiteness of the cost; see Remark~\ref{rem_Qdefinite} in \ref{prf_lem_stab}. 

It is well-known that the minimum of the cost function can be attained by the optimal {\it LQG controller} given by 
\begin{equation}\label{dyn_LQG}
    \rmsK {\bm K}_{\rm LQG} :
\left\{\H
\begin{array}{rcl}
    \xi(t+1) \HH &=&\HH  G_{\rm LQG} \xi(t) + H_{\rm LQG} y(t)\\
    u(t) \HH &=&\HH  F_{\rm LQG} \xi(t) 
\end{array}
\right.
\end{equation}
with $G_{\rm LQG} := A+BF_{\rm LQG}-H_{\rm LQG}C$, where $F_{\rm LQG}$ and $H_{\rm LQG}$ are determined by solving Riccati equations \cite{soderstrom2002discrete}. 
However, the main interest of this paper is to use a PGM (instead of Riccati equations or LMIs) to obtain the optimal policy \eqref{dyn_LQG}.
For this purpose, an appropriate parametrization of the dynamic controller \eqref{dyn_K} is needed, as discussed in the `Related Works'' section above. 
For example, the system matrices $(G,H,F)$ can be used to parametrize the controller. 
Unfortunately, $\rmsJ J(\rmsK {\bm K})$ as a function of $(G,H,F)$ has a complex landscape. For example, consider the case where $H = 0$, fixing $G$ to a Schur stable matrix and varying only $F$. In this case, even if $\|F\| \to \infty$, the transfer function of the controller remains zero, thus $\rmsJ J$ also remains unchanged. This example illustrates that $\rmsJ J$ is not coercive along certain directions in the $(G, H, F)$ space. As a consequence, even the convergence of a gradient algorithm to a stationary point cannot be theoretically guaranteed \cite{hu2023toward}. Therefore, in the next section, we introduce an alternative policy parametrization that leads to a favorable optimization landscape.

\section{Preliminary}

\subsection{IOH Representation of the System $\rms {\bm \Sigma}$}
\begin{definition}\label{def_1}
Let $\{u,y\}$ be the input-output signal of $\rms {\bm \Sigma}$ in \eqref{hat Sigma}. Given $L \in \mathbb N$, we refer to
\begin{equation}\label{def_IOH}
    z(t) := [([u]^{t-1}_{t-L})^{\top}, ([y]^{t-1}_{t-L})^{\top}]^{\top} \in \mathbb R^{n_z}, \quad t \geq L
\end{equation}
with $n_z := L(n_u + n_y)$ as the {\it input-output history of length $L$} or simply the {\it IOH}. 
Similarly, we refer to
\begin{equation}
    e(t) := [([w]^{t-1}_{t-L})^{\top}, ([v]^{t-1}_{t-L})^{\top}]^{\top} \in \mathbb R^{n_e}, \quad t \geq L
\end{equation}
with $n_e := L(n_w + n_v)$ as the noise history, and 
\begin{equation}\label{def_fullH}
    h(t) := [z^{\top}(t), e^{\top}(t)]^{\top} \in \mathbb R^{n_h}, \quad t \geq L
\end{equation}
with $n_h := n_z + n_e$ as the full history. 
\end{definition}

\begin{definition}\label{def_Lmeas} \cite{rugh1996linear}
An $n_{\eta}$-dimensional system $\eta(t+1) = A_{\eta}\eta(t) + B_{\eta}u(t), y(t)=C_{\eta}\eta(t)$ is said to be \it $L$-step observable if $\rank \mathcal O_L(A_{\eta},C_{\eta})=n_{\eta}$.
\end{definition}
The existence of $L \in \mathbb Z$ such that $\rms{\bm \Sigma}$ is $L$-step observable is equivalent to that $\rms{\bm \Sigma}$ is observable.

\begin{lem}\label{sys_lem}
Suppose the system $\rms {\bm \Sigma}$ in \eqref{hat Sigma} is $L$-step observable, and let $z$ be the IOH of length $L$ defined by \eqref{def_IOH}, and $h$ be the full history defined by \eqref{def_fullH}. Consider a new system ${\bm \Sigma}$ defined by 
\begin{equation}\label{def_sigma}
\hspace{-2mm} {\bm \Sigma} :  \left\{ \H
\begin{array}{rcl}
    h(t+1) &\HH = &\HH \Theta h(t) + \Pi_d d(t) + \Pi_u u(t)\\   
    y(t)  &\HH = &\HH \Psi h(t) + \Upsilon d(t)\\  
    z(t)  &\HH = &\HH \Gamma h(t) \\
\end{array}\HH \right. ,~ t \geq L
\end{equation}
where $d$ is defined in \eqref{def_d} and 
\begin{align}
    &\Theta \h := \h \left[
    \begin{array}{cccc}
     J_{n_u} & & & \\
     & J_{n_y} & & \\
     & & J_{n_w} & \\
     & & & J_{n_v}
    \end{array}
     \right] + 
    \left[
     \begin{array}{c}
    0  \\
    E_{n_y}C[M_1~M_2~M_3~M_4] \\
    0 \\
    0  
 \end{array}
 \right] \nonumber \\
& \Pi_d \h :=\h \left[
    \H \begin{array}{cc}
    0 \H & 0 \\
    0 \H & E_{n_y} \\
    E_{n_w} \H & 0 \\
    0 \H & E_{n_v} 
    \end{array}
 \H \right], \quad \Pi_u \h :=\h  \left[
    \H \begin{array}{c}
    E_{n_u}  \\
    0 \\
    0 \\
    0  
 \end{array}
\H \right] \nonumber \\
& J_n \h :=  \left[
    \begin{array}{cccc}
     0_{n\times n} & I_n & & \\
     &\HHH \ddots &\HHH \ddots &\HHH \\
    \hH &\HHH &\HHH  &\HHH I_n \\
    \hH &\HHH &\HHH &\HHH 0_{n \times n}
    \end{array}
     \right]  \in \h \mathbb R^{Ln \times Ln}, E_n \h :=  \left[
     \begin{array}{c}
    0_{n\times n}  \\
    \vdots  \\
    0_{n\times n} \\
    I_n  
 \end{array}
 \right]  \in \h \mathbb R^{Ln \times n} \nonumber \\
&\Psi \h :=\h C[M_1, M_2, M_3, M_4],~\Gamma \h:=\h [I_{n_z}, 0],~\Upsilon \h:=\h [0, I_{n_y}] \nonumber \\
&M_1 \h := \h {\mathcal R}_L(A,B)-A^L{\mathcal O}_L^{\dagger}(A,C){\mathcal H}_L(A,B,C) \nonumber \\
&M_2 \h := \h A^L{\mathcal O}_L^{\dagger}(A,C) \nonumber \\
&M_3 \h := \h {\mathcal R}_L(A,I)-A^L{\mathcal O}_L^{\dagger}(A,C){\mathcal H}_L(A,I,C) \nonumber \\
&M_4 \h := \h -A^L{\mathcal O}_L^{\dagger}(A,C).  \nonumber
\end{align}
If the initial state $h(L)$ of \eqref{def_sigma} coincides with the history of $\{u, y, w, v\}$ for $t \in \{0,\ldots, L-1\}$, then $\rms {\bm \Sigma}$ and ${\bm \Sigma}$ are equivalent in the sense that for any input sequences $\{u, w, v\}$ the output $y$ generated by \eqref{hat Sigma} and \eqref{def_sigma} coincide for each $t \geq L$. 
\end{lem}
\begin{proof}  
See \ref{prf_sys_lem}. 
\end{proof}
The IOH representation \eqref{def_sigma} has an important property in that the IOH vector $z$ (the first $n_z$ component of the full history $h$ of the plant $\rms{\bm \Sigma}$) can be directly measured by the controller. 

\subsection{IOH Representation of the Controller $\rmsK {\bm K}$}

\begin{lem}\label{lem2}
Given $L \in \mathbb N$, consider  
\begin{equation}\label{IOH_law}
    {\bm K}:~u(t) = Kz(t), \quad t \geq L, 
\end{equation}
where $z(t)$ is defined in \eqref{def_IOH} and $K=[K^u, K^y] \in \mathbb R^{n_u \times n_z}$ is partitioned as
\begin{align}
K^u = [K^u_L, \ldots, K^u_1], \; K^y=[K^y_L, \ldots, K^y_1] \label{Kiuy}     
\end{align}
with $K^u_i \in \mathbb R^{n_u \times n_u}$ and $K^y_i \in \mathbb R^{n_u \times n_y}$. Then, the following two statements hold. 
\begin{enumerate}
    \item[a)]  Suppose the dynamic controller $\rmsK {\bm K}$ in \eqref{dyn_K} is $L$-step observable. Then \eqref{dyn_K} and \eqref{IOH_law} are equivalent if 
    \begin{align}
        &\HHH \HHH K^u := FG^L\mathcal O_L^{\dagger}(G,F), \label{defKu_h} \\
        &\HHH \HHH K^y := F{\mathcal R}_L(G,H)-FG^L \mathcal O_L^{\dagger}(G,F){\mathcal H}_L(G,H,F). \label{defKy_h}
    \end{align}
    \item[b)]  Conversely, given \eqref{IOH_law}, construct a dynamic controller $\rmsK {\bm K}$ in \eqref{dyn_K} with \footnote{If $L = 1$, then $G = K_1^u$, $H = K_1^y$, $F = I$ and $\xi(0) = u(0)$. }
    \begin{align}\label{ABCD_hat}
     G \hspace{-0.5mm}=\hspace{-1mm} \left[
    \begin{array}{cccc}
    & & & K^u_L \\
    I  &  & & K^u_{L-1} \\
    & \ddots &  & \vdots \\
    & & I & K^u_1
    \end{array}
    \right],~H \hspace{-0.5mm}=\hspace{-1mm} \left[\hspace{-1.5mm}
    \begin{array}{c}
        K^y_L \\ K^y_{L-1} \\ \vdots \\ K^y_1
    \end{array}\hspace{-1.5mm}
    \right],~F \hspace{-0.5mm}=\hspace{-1mm} \left[\hspace{-1.0mm}
    \begin{array}{c}
    0 \\ \vdots \\ 0 \\ I
    \end{array}
    \hspace{-1.0mm}\right]^{\top}\hspace{-1mm}
\end{align}
and $\xi(0) = \mathcal O_L^{-1}(G, F)[I_{Ln_u}, -\mathcal H_L(G, H, F)] z(L)$, where $\mathcal O_L(G, F)$ is always invertible. Then, \eqref{dyn_K} and \eqref{IOH_law} are equivalent. 
\end{enumerate}
\end{lem}
\begin{proof} The proof is similar to Lemma 2 in \cite{sadamoto2024policy}.
\end{proof}

The triple $(G, H, F)$ in \eqref{ABCD_hat} is one realization based on the given ${\bm K}$ in \eqref{IOH_law}. Note that there exists a degree of freedom in choosing this triple, up to a coordinate transformation. This lemma implies that any dynamic controller of form \eqref{dyn_K} can be equivalently written as
\begin{align}
    u(t) = [K^u,~K^y,~0,~0]h(t) \label{structe}
\end{align}
with appropriate matrices $K^u$ and $K^y$. In other words, these matrices can be used to parameterize the class of dynamic control policies in the form \eqref{dyn_K}. Such a parametrization is model-free in the sense that reconstructing $\rmsK {\bm K}$ from $K$ does not require the knowledge of the plant model $(A,B,C)$. Therefore, we consider the following design strategy: 
\begin{enumerate}
    \item[a)] Find a desirable $K$ in \eqref{IOH_law}. 
    \item[b)] Transform the designed $K$ into $\rmsK {\bm K}$ using \eqref{dyn_K} and \eqref{ABCD_hat}. 
\end{enumerate}
Hereafter, we focus on design problem a), which is formulated in the next subsection. 

\subsection{IOH Representation of LQG problem}\label{Sec3C}
To reformulate the original LQG control problem as an equivalent control problem in terms of $K$ in \eqref{IOH_law},
we introduce an IOH representation of $\rmsJ J$ in \eqref{defJ_orig}. Given $L \in \mathbb N$ such that $\rms {\bm \Sigma}$ in \eqref{hat Sigma} is $L$-step observable, we define
\begin{equation}
    \label{cost}
     J(K) \h :=\h \lim_{T \rightarrow \infty}\h {\mathbb E}\h
    \left[\frac{1}{T} 
    \sum_{t=L}^{T}y^{\top}(t)Qy(t) + u^{\top}(t)Ru(t) \right]\H, 
\end{equation}
where $y(t)$ and $u(t)$ follow the dynamics of the closed-loop system $({\bm \Sigma}, {\bm K})$. 
From Lemmas \ref{sys_lem}–\ref{lem2}, it is evident that the closed-loop systems $(\rms {\bm \Sigma}, \rmsK {\bm K})$ and $({\bm \Sigma}, {\bm K})$ are equivalent. Therefore, we obtain the following: 
\begin{lem}\label{prp1}
Suppose $\rms {\bm \Sigma}$ is $L$-step observable. If $\rmsK {\bm K}$ in \eqref{dyn_K} and ${\bm K}$ in \eqref{IOH_law} are related by \eqref{ABCD_hat}, then
\begin{align}\label{equi_cost}
    \rmsJ J(\rmsK {\bm K}) = J(K). 
\end{align}
\end{lem}
\begin{proof}
See \ref{prf_lem3}. 
\end{proof}

In the following discussion, we introduce
\begin{align}
    \mathbb K := \{K~{\rm s.t.}~ \Theta_K := \Theta + \Pi_uK\Gamma~{\rm is~Schur~stable}\}, \label{def_thetaK}
\end{align}
and impose the following assumption:
\begin{assumption}\label{ass_noemp}
    The set $\mathbb K$ is not empty. 
\end{assumption}
Assumption \ref{ass_noemp} ensures that the system ${\bm \Sigma}$ in \eqref{def_sigma} is stabilizable by a partial-state feedback control \eqref{IOH_law}. Under this setting, we have the following theorem: 
\begin{theo}\label{opt_thm}
Suppose Assumptions \ref{ass_ABC}–\ref{ass_noemp} hold, and that $\rms {\bm \Sigma}$ in \eqref{hat Sigma} is $L$-step observable. Suppose also that the minimal realization of $\rmsK {\bm K}_{\rm LQG}$ defined by \eqref{dyn_LQG} is $L$-step observable. Define
\begin{equation}\label{defKstar}
 K_{\star} := {\rm argmin}_{K \in \mathbb K} J(K), 
\end{equation}
and let $\rmsK {\bm K}_{\star}$ be constructed from $K_{\star}$ using \eqref{dyn_K} and \eqref{ABCD_hat}. Then
\begin{equation}\label{opteq1}
J(K) \geq J(K_{\star}) = \rmsJ J(\rmsK {\bm K}_{\star}) = \rmsJ J(\rmsK {\bm K}_{\rm LQG})
\end{equation}
holds for any $K \in \mathbb R^{n_u \times n_z}$.
 Moreover when $n_u = 1$, the $L$-dimensional controller $\rmsK {\bm K}_{\star}$ is an optimal solution to \eqref{defJ_orig} in the form of \eqref{dyn_K}, i.e., 
\begin{align}
    \rmsK {\bm K}_{\star} = {\rm argmin}  \rmsJ J(\rmsK {\bm K})\quad {\rm s.t.}~ \rmsK {\bm K}~{\rm in}~(\ref{dyn_K})~{\rm with}~n_{\xi}=L. \label{opteq2}
\end{align}
\end{theo}
\begin{proof} 
See \ref{prf_thm1}. 
\end{proof}

In Theorem~\ref{opt_thm}, the condition that the minimal realization of $\rmsK {\bm K}_{\rm LQG}$ is $L$-step observable ensures the existence of an IOH gain $K \in \mathbb R^{n_u \times n_z}$ equivalent to $\rmsK {\bm K}_{\rm LQG}$. Note that this does not require $\rmsK {\bm K}_{\rm LQG}$ itself to be minimal. Therefore, Theorem 1 remains valid even when the LQG controller is non-minimal (i.e., the results of \cite{tang2021analysis,zheng2023benign} are not applicable). 

As \eqref{structe} is a {\it structured state-feedback policy} in the sense that the last two matrix gains in \eqref{structe} are forced to be zero, it is difficult to show that $J(K)$ is gradient dominant \cite{duan2023optimization}. However, in the numerical simulations presented later, the PGM initialized from 20 different IOH gains all converges to the LQG controller. Whether $J$ as a function of IOH gains possesses benign properties such as gradient dominance remains an open problem.

As a first step towards the convergence analysis of PGMs for the LQG problem, we next propose a vanilla PGM with a theoretical guarantee of convergence to $\mathcal O(\epsilon)$-stationary points of $J$. More specifically, we consider a relaxed version of the LQG problem, where small process noise is added to the IOH dynamics in \eqref{def_sigma}. We then show that the PGM converges to stationary points of this relaxed problem. As the covariance of the additional noise decreases, the stationary points become closer to those of the original LQG cost, ensuring that the PGM converges to $\mathcal O(\epsilon)$-stationary points of $J$. To the best of our knowledge, this work is the first to analyze the convergence of PGMs for the LQG problem.

\begin{rem}
A part of the results presented in this section can be found in~\cite{al2022behavioral}, where Lemma~\ref{sys_lem}, Lemma~\ref{lem2}~(a), and the relation $J(K_{\star}) = \rmsJ J(\bm{K}_{\rm LQG})$ for the case $L = n_x$ were established as Lemmas 5.2--5.4. In contrast, this paper additionally establishes Lemma~\ref{lem2}~(b) and extends all relevant results to more general settings where $L \neq n_x$. This generalization enables the construction of low-dimensional controllers via the proposed PGM when $L < n_x$. Moreover, even when $L > n_x$, the theoretical guarantees including the subsequent convergence analysis remain valid. This flexibility is particularly important when extending the results to model-free PGMs, where the true system dimension is unknown due to the lack of model knowledge. In such cases, overestimating $L$ serves as a practical design guideline. Therefore, clearly establishing that the theoretical guarantees remain valid for $L > n_x$ is crucial for model-free PGMs. 
\end{rem}

\section{PGM for A Relaxed LQG Problem}
\subsection{Motivation for Relaxation}\label{sec_non_coer}
First of all, we show that $J(K)$ in \eqref{cost} is not coercive in general. To see this, let us consider a particular case where $\rms{\bm \Sigma}$ in \eqref{hat Sigma} is a stable, two-input, single-output system (i.e., $n_u = 2$ and $n_y = 1$). Furthermore, suppose that $\rms{\bm \Sigma}$ is $L$-step observable with $L = 3$. For this system, consider an IOH control law $K$ in \eqref{IOH_law} where $K^{y} = 0$ and
\begin{align}
    K_1^u = \begin{bmatrix}
        0.9 & a \\
        0 & 0.8
    \end{bmatrix}, \quad
    K_2^u = K_3^u = 0, \quad a \in \mathbb{R}.
\end{align}
Then, $H$ in \eqref{ABCD_hat} becomes $H = 0$. Thus, the closed-loop system $(\rms{\bm \Sigma}, \rmsK{\bm K})$, where $\rmsK{\bm K}$ is constructed using \eqref{dyn_K} and \eqref{ABCD_hat}, becomes a cascade system in which $\rmsK{\bm K}$ is located upstream and $\rms{\bm \Sigma}$ is located downstream. Therefore, the eigenvalues of the closed-loop system are those of $A$ and $G$. Here, the eigenvalues of $G$ in \eqref{ABCD_hat} are $\{0.9, 0.8, 0, 0, 0, 0\}$. Hence, $(\rms{\bm \Sigma}, \rmsK{\bm K})$ is stable regardless of the value of $a$. Since the transfer function of $\rmsK{\bm K}$ is zero for any $a$, the cost $\rmsJ J(\rmsK{\bm K})$ in \eqref{defJ_orig} remains unchanged with respect to $a$. Therefore, from Lemma~\ref{prp1}, $J(K)$ also remains unchanged with respect to $a$, which shows that $J(K)$ is not coercive in general. To enforce coerciveness, we consider a relaxed version of the LQG problem in the following section.

\subsection{Relaxation of the LQG Problem}
In the sequel, instead of ${\bm \Sigma}$ in \eqref{def_sigma}, we consider a hypothetical model in which the state of ${\bm \Sigma}$ is influenced not only by $d$ and $u$, but also by extra small process noise described as 
\begin{align}\label{perturb}
    \delta(t) \sim \mathcal{N}(0, {\epsilon} I), \quad 0 < {\epsilon} \ll 1. 
\end{align}
The resulting hypothetical model is described as follows: 
\begin{equation}\label{def_sigma_relaxed}
\hspace{-2mm} {\bm \Sigma}^{\epsilon} :  \left\{ \H
\begin{array}{rcl}
    h^{\epsilon}(t+1) &\HH = &\HH \Theta h^{\epsilon}(t) + \Pi_d d(t) + \Pi_u u(t) + \delta(t)\\   
    y^{\epsilon}(t)  &\HH = &\HH \Psi h^{\epsilon}(t) + \Upsilon d(t)\\  
    z^{\epsilon}(t)  &\HH = &\HH \Gamma h^{\epsilon}(t), \\
\end{array}\HH \right.
\end{equation}
where $t \geq L$. The system matrices such as $\Theta$ are the same as those defined in Lemma~\ref{sys_lem}. The process noise $\delta$ will play a role in the theoretical guarantee of the convergence, as shown later. Note that $z^{\epsilon}$ defined here does not satisfy \eqref{def_IOH}, i.e., $z^{\epsilon}(t) \ne [([u]^{t-1}_{t-L})^{\top}, ([y^{\epsilon}]^{t-1}_{t-L})^{\top}]^{\top}$. For this hypothetical model, we consider a control law with the same structure as \eqref{IOH_law}, i.e., 
\begin{equation}\label{IOH_law_relaxed}
    {\bm K}^{\epsilon}:~u(t) = Kz^{\epsilon}(t), \quad t \geq L.
\end{equation}
For the closed-loop system $({\bm \Sigma}^{\epsilon}, {\bm K}^{\epsilon})$, similarly to \eqref{cost}, we define a cost function defined as 
\begin{equation}
    \label{cost_relax}
    \hH J^{\epsilon}(K) \h :=\h \lim_{T \rightarrow \infty}\h {\mathbb E}\h
    \left[\frac{1}{T} 
    \sum_{t=L}^{T}y^{\epsilon\top}(t)Qy^{\epsilon}(t) + u^{\top}(t)Ru(t) \right]\H.
\end{equation}
The relationship between $J^{\epsilon}$ and the original LQG cost $J$ in \eqref{cost} is given as follows. For any $K \in \mathbb{K}$, where $\mathbb{K}$ is defined in \eqref{def_thetaK}, it follows that
\begin{align}\label{Jerr}
 J^{\epsilon}(K) =J(K) + \gamma_K{\epsilon}, \quad \gamma_K := \|\Omega_K ({\sf z}I - \Theta_K)^{-1}\|_{H_2}^2, 
\end{align}
where $\Omega_K := [\Psi^{\top}Q^{\frac{1}{2}}, \Gamma^{\top}K^{\top}R^{\frac{1}{2}}]^{\top}$ and $\Theta_K$ is defined in \eqref{def_thetaK}. The derivation of \eqref{Jerr} is shown in  \ref{prf_Jerr}. It follows from \eqref{Jerr} that if a stationary point of $J^{\epsilon}$ is obtained when $\epsilon \ll 1$, it can be expected to be a good approximation to a stationary point of $J$. To describe this precisely, we introduce the following definition.
\begin{definition}
 Consider $J$ in \eqref{cost}. We say that $\overline{K} \in \mathbb K$ is {\it $\mathcal O(\epsilon)$-stationary point} if 
 \begin{align}\label{def:eps_stationary}
        \|\nabla J(\overline{K})\|_F \leq \mathcal O({\epsilon}), \quad   \mathcal O(\epsilon) := \beta \epsilon
    \end{align}
    holds, where $\beta$ is constant parameter without depending on $\overline{K}$ and $\epsilon$. 
\end{definition}
In the remainder of this paper, we show that the gradient algorithm for $K$ introduced in the next section converges to one of $\mathcal O(\epsilon)$-stationary points.

\subsection{PGM}
For obtaining a stationary point of $J^{\epsilon}$ in \eqref{cost_relax}, we consider a vanilla PGM described as
\begin{equation}
\label{gd}
    K_{i+1} = K_{i}-\alpha\nabla J^{\epsilon}(K_i). 
\end{equation}
Here, $i \geq 0$ is the iteration number and $\alpha > 0$ is the step-size parameter. The gradient $\nabla J^{\epsilon}$ is expressed as follows.

\begin{lem} \label{lem_nablaJ}
Let ${\bm \Sigma}^{\epsilon}$ in \eqref{def_sigma_relaxed} be given. For each $K \in \mathbb K$, let $\Phi_K \geq 0$ and $Y_K > 0$ be the solutions to
\begin{align}
    &\Theta_K^{\top} \Phi_K \Theta_K - \Phi_K + \Psi^{\top}Q\Psi  + \Gamma^{\top}K^{\top}RK\Gamma = 0 \label{lyap} \\
    &\Theta_K Y_K \Theta^{\top}_K - Y_K + \Pi_d V_d \Pi_d^{\top} + {\epsilon}I = 0, \label{EKH2}
\end{align}
respectively. Define 
\begin{align}\label{EKH1}
 W_{K}:= (\Pi_u^{\top}\Phi_K\Pi_u + R)K\Gamma+\Pi_u^{\top}\Phi_K\Theta. 
\end{align}
Then, the gradient of $J^{\epsilon}(K)$ defined in \eqref{cost_relax} is given by
\begin{equation}\label{nablaJep}
\nabla J^{\epsilon}(K)=2W_K Y_{K}\Gamma^{\top}.  
\end{equation}
\end{lem}
\begin{proof}
See \ref{prf_lem_nablaJ}. 
\end{proof}

\begin{algorithm}[t]
\caption{PGM for the relaxed LQG problem}\label{alg1}
\begin{algorithmic}[1]
\STATE \textbf{Input:} $Q,R>0$ in \eqref{defJ_orig}; $L \in \mathbb N$ such that $\rms {\bm \Sigma}$ is $L$-step observable; $\alpha>0$; $\epsilon>0$; $K_0 \in \mathbb K$ where $\mathbb K$ is defined in \eqref{def_thetaK}
\FOR{$i = 0,1,2,\dots$ \textbf{until convergence}}
    \STATE Compute $\nabla J^{\varepsilon}(K_i)$ in \eqref{nablaJep}, and update $K_{i+1}$ by \eqref{gd}. 
\ENDFOR
\STATE \textbf{Output:} $\rmsK {\bm K}$ in \eqref{dyn_K} and \eqref{ABCD_hat}, where $K \leftarrow K_i$
\end{algorithmic}
\end{algorithm}

In this lemma, $Y_K$ is positive-definite due to the process noise \eqref{perturb}, which is crucial to guarantee the coerciveness of $J^{\epsilon}$. 
Algorithm \ref{alg1} summarizes the PGM we consider in this paper. Next, we analyze its convergence property.

\subsection{Convergence Analysis}

\begin{lem}\label{lem_stab}
 Assume $\rms {\bm \Sigma}$ is $L$-step observable. Then, $J^{\epsilon}(K)$ in \eqref{cost_relax} is bounded iff $K \in \mathbb K$, where $\mathbb K$ are defined in \eqref{def_thetaK}. 
\end{lem}
\begin{proof}
    See \ref{prf_lem_stab}. 
\end{proof}

From this lemma and \eqref{Jerr}, it follows that if $J^{\epsilon}(K) < \infty$, then $J(K)$ is also bounded. To establish the coercive property, we introduce the following notation and assumption. For each $\overline K\in\partial\mathbb K$, let $\lambda^{(1)},\dots,\lambda^{(w)}$
denote the distinct eigenvalues of $\Theta_{\overline K}$ such that $|\lambda^{(s)}|=1$.
For each $s\in\{1,\dots,w\}$, let $\mathcal C^{(s)} \subset \mathbb C$ denote a closed positively oriented curve such that
$\lambda^{(s)}\in\rm{int}(\mathcal C^{(s)})$ and no other eigenvalues lie inside or on $\mathcal C^{(s)}$. Let $\mathcal B_{\eta}(\overline{K}) := \{K \in \mathbb R^{n_u \times n_h} ~|~ \|K -\overline{K}\|_F < \eta \}$. 
\begin{assumption}\label{ass_NOBIF}
For each $\overline K\in\partial\mathbb K$ and each $s\in\{1,\dots,w\}$, there exists $\eta^{(s)}>0$ such that, for every $K\in\mathcal B_{\eta^{(s)}}(\overline K)$, no eigenvalue of $\Theta_{K}$ lies on $\mathcal C^{(s)}$ and the eigenvalues of $\Theta_{K}$ lying inside $\mathcal C^{(s)}$ consist of a single eigenvalue $\lambda^{(s)}_K$ whose algebraic multiplicity $m^{(s)}$ is independent of $K$ within $\mathcal B_{\eta^{(s)}}(\overline K)$. 
\end{assumption}
This assumption implies that no bifurcations or mergers of the eigenvalues of $\Theta_K$ occur on the boundary $\partial\mathbb K$. Thus, it is a mild assumption and we believe it does not significantly restrict the scope of the paper. Under this assumption, we show that $J^{\epsilon}(K)$ is coercive, which is an important property for optimizing $K$.

\begin{lem}\label{lem_coer}
 Assume $\rms {\bm \Sigma}$ is $L$-step observable and Assumption \ref{ass_NOBIF} holds. Then, $J^{\epsilon}(K)$ in \eqref{cost_relax} is coercive, i.e., for any sequence $\{K_i\}_{i=0}^{\infty} \subseteq \mathbb K$, we have
 \begin{align}
     J^{\epsilon}(K_i) \rightarrow \infty,~ {\rm if}~\lim_{i\rightarrow \infty}K_i \in\partial \mathbb K~{\rm or}~\lim_{i\rightarrow \infty}\|K_i\| = \infty,
 \end{align}
 where $\mathbb K$ is defined in \eqref{def_thetaK}. 
\end{lem}
\begin{proof}
See \ref{prf_lem_coer}. 
\end{proof}

Since $J^{\epsilon}(K)$ is coercive and clearly has a lower bound, the sublevel set
\begin{align}\label{sublevel_set}
    \mathbb{K}_{c} := \{K \in \mathbb{K}~|~ J^{\epsilon}(K) \leq c\}
\end{align} 
is compact if Assumption \ref{ass_NOBIF} holds. Next, we show that $J^{\epsilon}(K)$ is smooth over this sublevel set. 
\begin{lem}\label{lem_hesian}
Consider $J^{\epsilon}(K)$ in \eqref{cost_relax}. Impose Assumption \ref{ass_NOBIF}. Given $c \geq J^{\epsilon}(K_0)$, consider $\mathbb K_c$ in \eqref{sublevel_set}. Then, there exists $q > 0$ that satisfies:  
\begin{align}
    \|\nabla^2 J^{\epsilon}(K)\| \leq q, \quad \forall K \in \mathbb K_{c}. \label{q_def}
\end{align}
\end{lem}
\begin{proof}
See \ref{prf_lem_hesian}. 
\end{proof}

Smoothness over a sublevel set is known for LQR \cite{fazel} and static output-feedback control \cite{duan2023optimization} problems. Lemma \ref{lem_hesian} shows that the relaxed cost as a function of the IOH gain has a similar property. With coerciveness and smoothness over a compact sublevel set established, we are now ready to present our main convergence result.

\begin{theo}\label{linear_convergence}
Suppose $\rms {\bm \Sigma}$ in \eqref{hat Sigma} is $L$-step observable. Impose Assumption \ref{ass_NOBIF}. Let $q$ be the parameter satisfying \eqref{q_def} and set $\alpha \in (0, 2/q)$. Then the PGM update rule \eqref{gd} converges to a stationary point of $J^{\epsilon}$, i.e.,  
\begin{equation}\label{statio_Je}
    \lim_{i \to \infty} \| \nabla J^{\epsilon}(K_i)\|_F = 0.
\end{equation}
Furthermore, it is also an $\mathcal O({\epsilon})$-stationary point of $J$, i.e.,  
\begin{equation}\label{statio_Jori}
    \lim_{i \to \infty} \| \nabla J(K_i)\|_F \leq \mathcal O(\epsilon).
\end{equation}
\end{theo}
\begin{proof}
See \ref{prf_linear_convergence}. 
\end{proof}
This theorem guarantees that Algorithm 1 with $\alpha \in (0, 2/q)$ finds an $\mathcal{O}(\epsilon)$-stationary point of $J$, i.e., $\overline{K} \in \mathbb{K}_c$ such that $\|\nabla J(\overline{K})\|_F \leq \mathcal{O}(\epsilon)$. 
While the theoretical result ensures only convergence to epsilon stationary points, numerical simulations suggest that Algorithm 1 globally converges to the vicinity of the globally optimal controller of the original problem.

\section{Numerical Simulation}
In this section, we demonstrate the effectiveness of Algorithm 1 using a numerical example. Consider $^{\rm s}{\bf \Sigma}$ in \eqref{hat Sigma} with
\begin{align}
\hspace{-0mm}
 &A \hspace{-0.5mm}=\hspace{-0.5mm} \begin{bmatrix}
 0.7349 &  0.1195 & 0.3545\\
    0.08005 & 0.961 &    -0.1506\\
    0.3654 & -0.1217 & 0.5076 
\end{bmatrix}\hspace{-1mm}, ~
B \hspace{-0.5mm}=\hspace{-0.5mm} \begin{bmatrix}
 -0.1158 \\
   0\\
-0.5297
\end{bmatrix} \nonumber \\
&C \hspace{-0.5mm}=\hspace{-0.5mm} 
\begin{bmatrix}    
-0.2326 & -0.5851 & 0.9771 \\
-0.1116 & 0 & 0.6755
\end{bmatrix}. \nonumber
\end{align}
We set $V_w = 0.1I_3$, $V_v = 0.1$ in \eqref{def_d}, and $Q=100I_2$, $R=10$ in \eqref{defJ_orig}. Under this setting, the optimal controller $\rmsK {\bm K}_{\rm LQG}$ is given by \eqref{dyn_LQG}, with 
\begin{align}
    H_{\rm LQG} \hspace{-0.5mm}=\hspace{-0.5mm} 
    \begin{bmatrix} 
-0.0422 & 0.6096 \\
-0.5763 & 0.3403 \\
0.1703 & 0.4165
\end{bmatrix},~ F_{\rm LQG} \hspace{-0.5mm}=\hspace{-0.5mm} 
\begin{bmatrix}    
0.3031\\
-1.0174\\
0.8212
\end{bmatrix}^{\top}. \nonumber 
\end{align}
An IOH representation  $K_{\rm LQG}$ of $\rmsK {\bm K}_{\rm LQG}$ can be constructed using claim a) of Lemma \ref{lem2}. In the following experiments, we used $\alpha = 1.0\times10^{-3}$ and ${\epsilon} = 10^{-8}$. 

First, we set $L = 3$.
Since the optimal controller $\rmsK {\bm K}_{\rm LQG}$ has a three dimensional state space (i.e., $n_{\xi} = 3$ in \eqref{dyn_K}), both $\rms {\bm \Sigma}$ and $\rmsK {\bm K}_{\rm LQG}$ are  $L$-step observable in this case.
To simulate Algorithm \ref{alg1}, we randomly generated initial gains $K_0$ with $\|K_0\|_F = 1$ from the uniform distribution, from which 20 stabilizing gains were randomly selected. For each of these initial gains, we applied Algorithm \ref{alg1} to generate a sequence $K_i$ for $i=0, \ldots, 10^5$. In Fig.~\ref{fig_JL3}, the progress $\rmsJ J(\rmsK {\bm K}_i)$ is plotted with a colored solid line for each of $20$ independent experiments. For comparison, $\rmsJ J(\rmsK{\bm K}_{\rm LQG})$ is plotted with a black dotted line in the same figure.
In Fig.~\ref{fig_BodeL3}, the Bode plots of the obtained 20 controllers at $i = 10^5$ are shown by colored solid lines and those of $\rmsK {\bm K}_{\rm LQG}$ are shown by black dotted lines. These figures indicate the convergence of Algorithm \ref{alg1} to the {\it optimal controller} regardless of the choice of the initial gains.

\begin{figure}[t]
\centering
    \includegraphics[width=1.0\columnwidth]{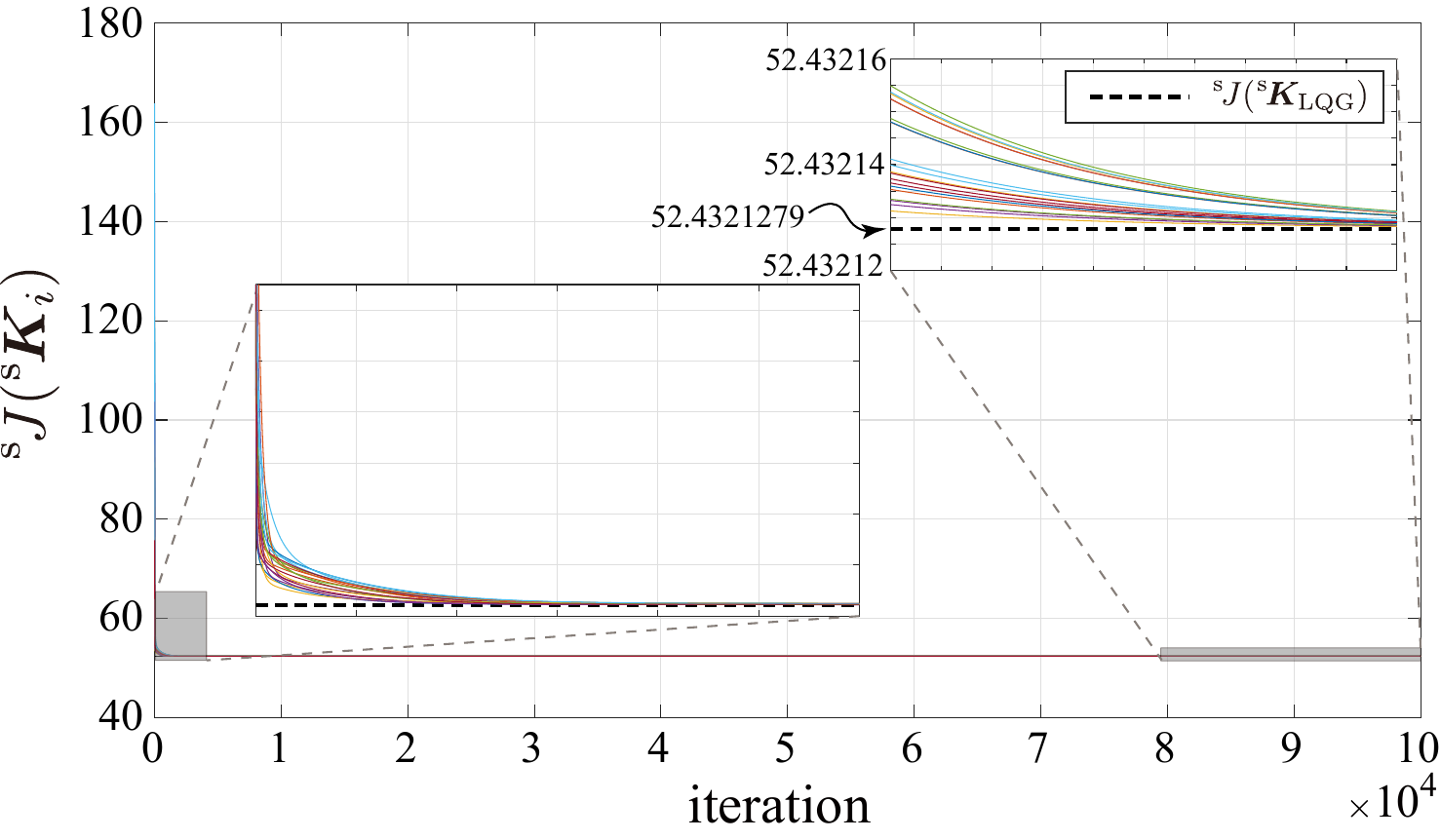} 
    \caption{The colored solid lines show 20 variations of $\rms J(\rmsK {\bm K}_i)$ for iteration $i$ when $L=3$ whereas the black dotted line shows $\rms J(\rmsK {\bm K}_{\rm LQG})$.} 
    \label{fig_JL3} 
\end{figure}
\begin{figure}[t]
\centering
    \includegraphics[width=1.0\columnwidth]{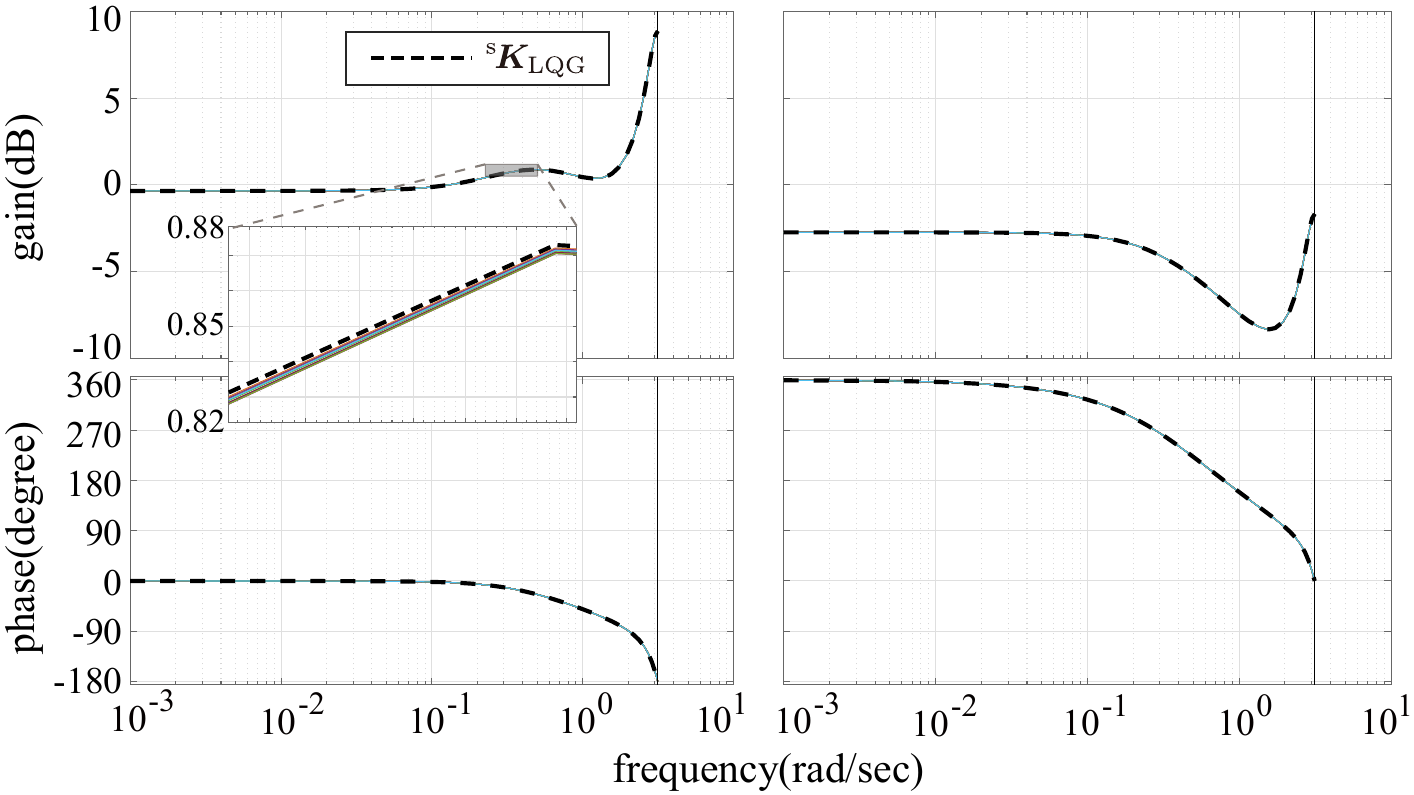} 
    \caption{The colored solid lines show the Bode diagrams of the 20 designed controllers $\rmsK {\bm K}_{10^5}$ when $L=3$. The black dotted line shows that of the true LQG controller $\rmsK {\bm K}_{\rm LQG}$.} 
    \label{fig_BodeL3} 
\end{figure}
\begin{figure}[t]
\centering
    \includegraphics[width=1.0\columnwidth]{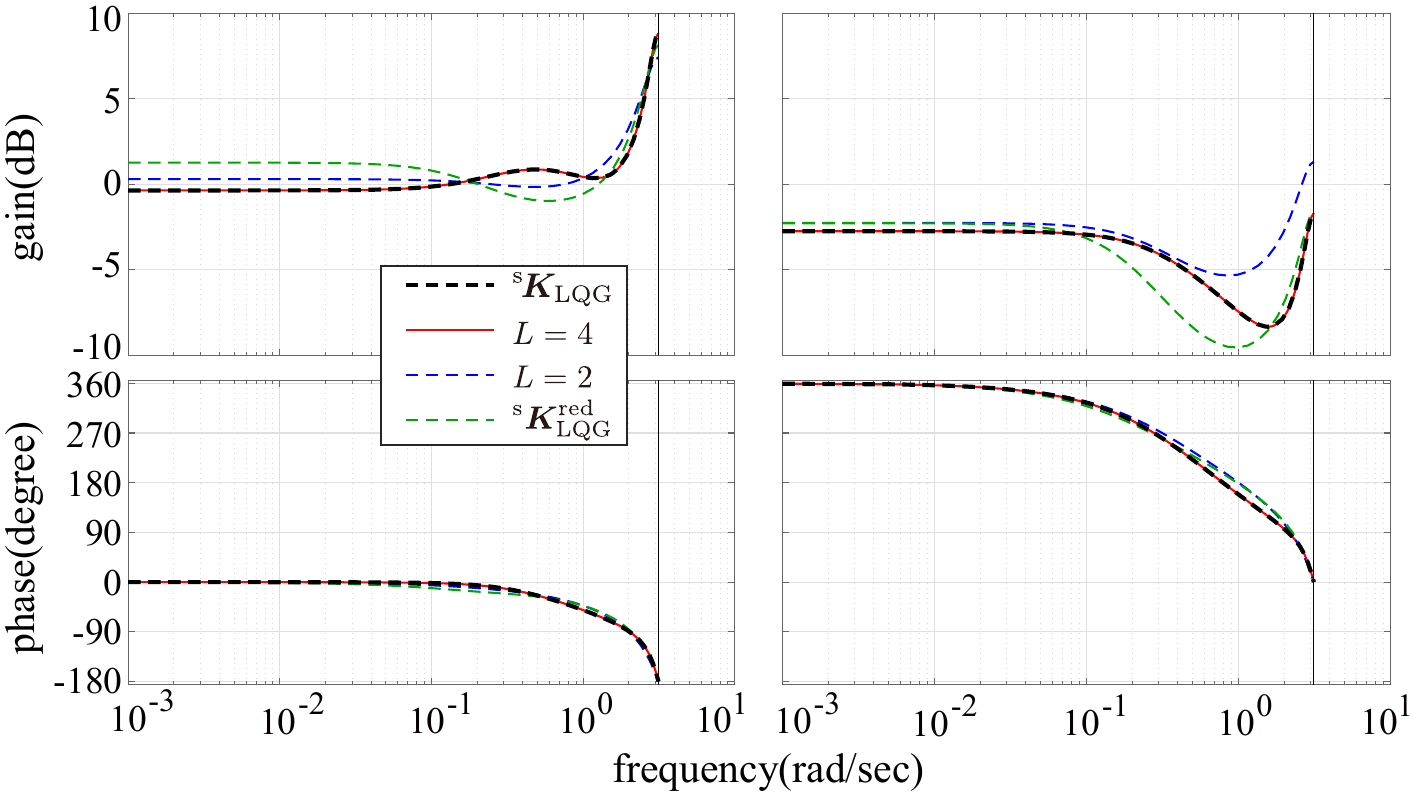} 
    \caption{The red solid line and blue dash-dotted line show the bode diagram of a designed four-dimensional $\rmsK {\bm K}_{10^5}$ and two-dimensional $\rmsK {\bm K}_{10^5}$, respectively, and the green dashed line shows that of $\rmsK {\bm K}_{\rm LQG}^{\rm red}$.} 
    \label{fig_BodeL24} 
\end{figure}
\begin{figure}[t]
\centering
    \includegraphics[width=1.0\columnwidth]{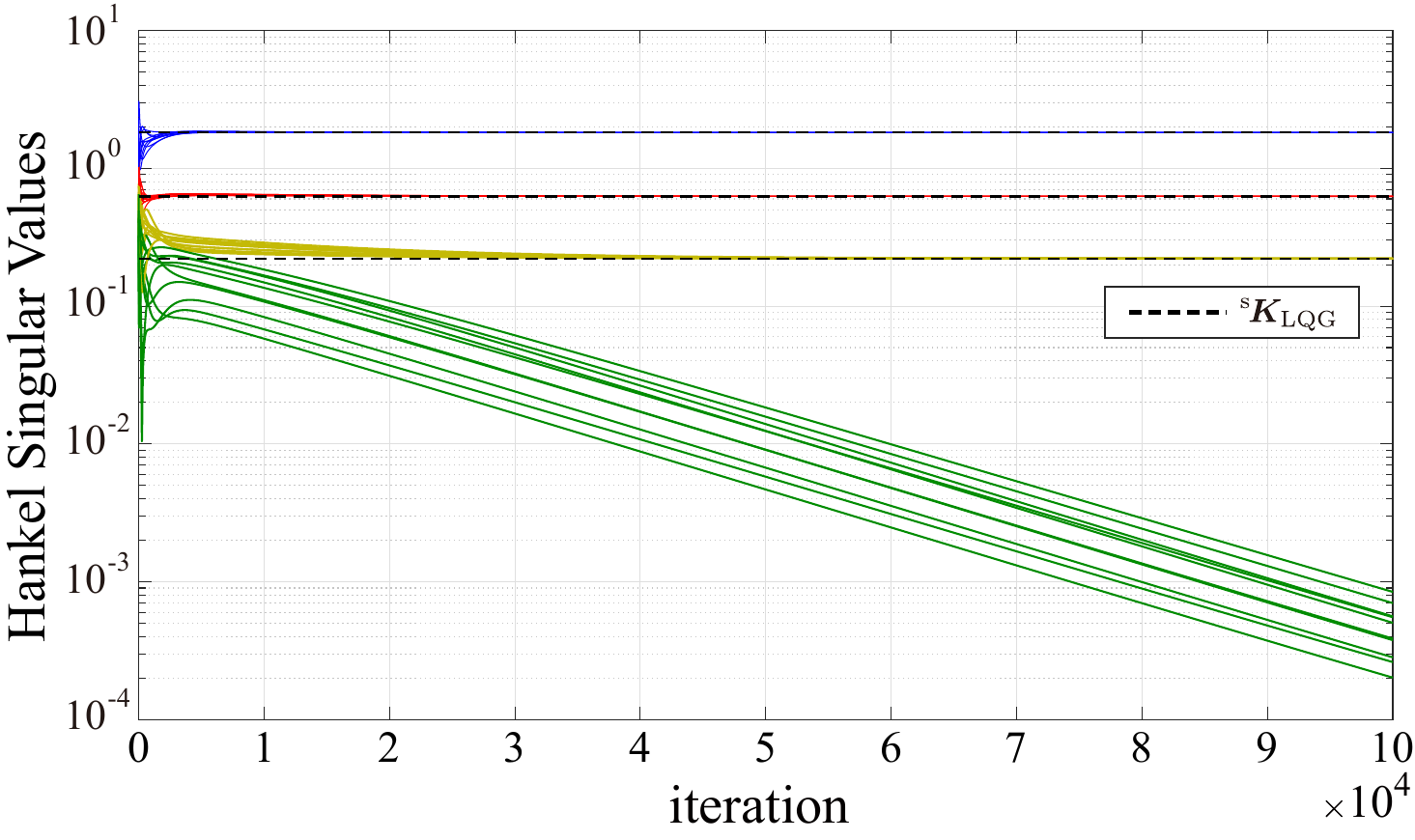} 
    \caption{The colored lines show the four Hankel singular values of $\rmsK {\bm K}_{i}$ for iteration $i$ when $L=4$. Black dotted lines show the three Hankel singular values of $\rmsK {\bm K}_{\rm LQG}$. } 
    \label{fig_hankel} 
\end{figure}

Next, we consider $L = 4$. In this case, the dimension of the dynamic controller to be designed is four.
Since $\rmsK {\bm K}_{\rm LQG}$ is three-dimensional, there is one extra redundant dimension.
The solid red line in Fig.~\ref{fig_BodeL24} shows the Bode diagram of the four-dimensional controller obtained after $i = 10^5$ iterations of  Algorithm \ref{alg1}.
For comparison, the Bode diagram of $\rmsK {\bm K}_{\rm LQG}$ is shown with the black dotted line. We observe a close match in the frequency responses of $\rmsK{\bm K}_{\rm LQG}$ and the synthesized four-dimensional controller. 
To better depict the behavior of Algorithm \ref{alg1}, Fig.~\ref{fig_hankel} shows the four Hankel singular values of $\rmsK{\bm K}_i$ (shown in blue, red, yellow, and green) over the iteration for $10$ different initial conditions. 
We observe that the first three converge to those of $\rmsK{\bm K}_{\rm LQG}$ (dotted black lines) whereas the fourth diminishes to zero. This result indicates that Algorithm \ref{alg1} not only recovers the optimal controller when overparameterized, but also effectively discards the redundant dimension by driving the corresponding Hankel singular value to zero.

We also consider the case with $L=2$ in Fig.~\ref{fig_BodeL24}.
In this case, even though $\rms {\bm \Sigma}$ is $L$-step observable, $\rmsK {\bm K}_{\rm LQG}$ is not. Since the dynamic controller
$\rmsK {\bm K}$ obtained by Algorithm \ref{alg1} under this setting is only two-dimensional, it is inevitably suboptimal. In Fig.~\ref{fig_BodeL24}, the Bode diagram of the two-dimensional controller $\rmsK {\bm K}_i$ obtained after a sufficient number of iterations ($i = 10^5$) is indicated by the blue dashed-dotted line. 
The cost attained by this controller is $\rmsJ J(\rmsK {\bm K}_{10^5}) = 52.5565$, which is slightly worse than the global optimum $\rmsJ J(\rmsK{\bm K}_{\rm LQG}) = 52.432179$. 
One of the well-known methods for constructing such reduced order controllers is controller reduction, where we first construct the $n_x$-dimensional LQG controller for an $n_x$-dimensional plant, and then reduce it using model reduction techniques such as the balanced truncation \cite{antoulas2005approximation}. In the numerical example, let ${\bm K}_{\rm LQG}^{\rm red}$ denote the two-dimensional controller constructed by this controller reduction approach. The cost for this controller is $\rmsJ J({\bm K}_{\rm LQG}^{\rm red}) = 53.1295$, which is slightly larger than $\rmsJ J(\rmsK {\bm K}_{10^5}) = 52.5565$. This indicates that Algorithm \ref{alg1} can find a better solution than the well-known controller reduction approach. 


\section{Conclusion}
Using the IOH representations of dynamic output feedback controllers, we provided a PGM-based LQG controller design algorithm that is convergent to $\mathcal O(\epsilon)$-stationary points. Numerical simulations suggested that the algorithm can find the optimal LQG controller and outperforms standard controller reduction techniques. Future research topics include further convergence analysis, theoretical investigation of global convergence guarantees, and development of model-free implementations of the proposed PGM. 

\appendix

\section{Proof of Lemma \ref{sys_lem}}\label{prf_sys_lem}
To simplify the notation, we denote 
$\mathcal R_L(A,B)$, 
$\mathcal R_L(A,I)$, 
$\mathcal O_L(A,C)$, 
$\mathcal H_L(A,B,C)$, 
$\mathcal H_L(A,I,C)$, as
$\mathcal R^u$, 
$\mathcal R^w$, 
$\mathcal O$,
$\mathcal H^u$, 
$\mathcal H^w$, respectively. The third equation in \eqref{def_sigma} is obvious. It follows from \eqref{hat Sigma} that
\begin{align}
\HH \H {[y]}_{t-L}^{t-1} & \h =\h \mathcal Ox(t-L) \h+\h \mathcal H^u[u]_{t-L}^{t-1}\h+\h \mathcal H^w[w]_{t-L}^{t-1} \h+\h [v]_{t-L}^{t-1},   \label{sys_dyn_y} \\
\HH \H x(t) & \h =\h A^Lx(t-L) + \mathcal R^u[u]_{t-L}^{t-1} + \mathcal R^w[w]_{t-L}^{t-1}.  \label{sys_dyn_y1}
\end{align}
Since $\rms {\bm \Sigma}$ is $L$-step observable, there exists $\mathcal O^{\dagger}$ such that $\mathcal O^{\dagger}\mathcal O = I$. Therefore, \eqref{sys_dyn_y} implies $x(t-L) = {\mathcal O}^{\dagger}\left({[y]}_{t-L}^{t-1} - [v]_{t-L}^{t-1} - \mathcal H^u[u]_{t-L}^{t-1} - \mathcal H^w[w]_{t-L}^{t-1}\right)$. By substituting this into \eqref{sys_dyn_y1}, we obtain  
\begin{align}
  x(t) = M_{12} z(t) + M_{34} e(t),  \label{x_IOH_rep}
\end{align}
where $M_{12} := [M_1, M_2]$ and $M_{34} := [M_3, M_4]$. By substituting this into the output equation in \eqref{hat Sigma}, we obtain 
\begin{align}\label{x_IOH}
y(t) = CM_{12} z(t) + CM_{34}e(t) + v(t)
\end{align}
which coincides with the second equation in \eqref{def_sigma}. Furthermore, from the definition of $z$ in \eqref{def_IOH}, the dynamics of $z$ is described as
\begin{align}
 z(t+1) &= \left[\begin{array}{c}
\left[\begin{array}{c}
[0, I][u^{\top}(t-L), ([u]^{t-1}_{t-L+1})^{\top}]^{\top} \\ u(t)
\end{array}
 \right] \\
\left[\begin{array}{c}
[0, I][y^{\top}(t-L), ([y]^{t-1}_{t-L+1})^{\top}]^{\top}  \\ y(t)
\end{array}
 \right]
\end{array}
 \right] \nonumber \\
 & = \left[\H \begin{array}{c}
J_{n_u}[u]^{t-1}_{t-L} \\
J_{n_y}[y]^{t-1}_{t-L} 
\end{array}
 \H\right] + \left[\H\begin{array}{c}
E_{n_u} \\ 0
\end{array}
 \H\right] u(t) + \left[\H\begin{array}{c}
 0 \\ E_{n_y}
\end{array}
 \H\right] y(t). \nonumber 
\end{align}
Similarly, we have 
\begin{align}
 e(t+1) = \left[\H \begin{array}{c}
J_{n_w}[w]^{t-1}_{t-L} \\
J_{n_v}[v]^{t-1}_{t-L} 
\end{array}
 \H\right] + \left[\H\begin{array}{c}
E_{n_w} \\ 0
\end{array}
 \H\right] w(t) + \left[\H\begin{array}{c}
 0 \\ E_{n_v}
\end{array}
 \H\right] v(t). \nonumber 
\end{align}
By combining these two equations and \eqref{x_IOH}, the first equation in \eqref{def_sigma} follows. This completes the proof.
\qed

\section{Proof of Lemma \ref{prp1}}\label{prf_lem3}
From Lemmas 1-2, the pair $\{u,y\}$ of the closed-loop  $({\bm \Sigma}, {\bm K})$ is identical to that of $(\rms {\bm \Sigma}, \rmsK {\bm K})$ for any triple $\{x(0), w, v\}$ and for $t \geq L$. Furthermore, 
$\rmsJ J(\rmsK {\bm K}) = \lim_{T \rightarrow \infty} \mathbb E \left[\frac{1}{T} \sum_{t = L}^{T} y^{\top}(t)Qy(t) + u^{\top}(t)Ru(t) \right]$, which is the same as the RHS of \eqref{cost}. This completes the proof. 
\qed

\section{Proof of Theorem \ref{opt_thm}}\label{prf_thm1}
First, we show \eqref{opteq1}. Without loss of generality, we assume that $\rmsK {\bm K}_{\rm LQG}$ is minimal. If not, by replacing the following proof with the minimal realization of $\rmsK {\bm K}_{\rm LQG}$, we have the same claim. As $\rmsK {\bm K}_{\rm LQG}$ is $L$-step observable, the claim a) in Lemma \ref{lem2} holds for $G \leftarrow G_{\rm LQG}$, $H \leftarrow H_{\rm LQG}$, and $F \leftarrow F_{\rm LQG}$. Denoting the resulting $K$  by $K_{\rm LQG}$, from Lemma \ref{prp1}, we obtain $\rmsJ J(\rmsK {\bm K}_{\rm LQG}) = J(K_{\rm LQG})$. As $\rmsK {\bm K}_{\rm LQG}$ is globally optimal, there is no $K$ such that $J(K) < J(K_{\rm LQG})$. Thus, $J(K_{\rm LQG}) = J(K_{\star})$ holds. This shows \eqref{opteq1}. Next, we show \eqref{opteq2}. Let $\rmsK {\bm K}^{L}_{\star}$ be an $L$-dimensional optimal solution to \eqref{defJ_orig} in the form of \eqref{dyn_K}. To this end, without loss of generality, we assume that $\rmsK {\bm K}^{L}_{\star}$ is observable. If not, by replacing the following proof with the minimal realization of $\rmsK {\bm K}^{L}_{\star}$, we have the same claim. Then, since $\rmsK {\bm K}^{L}_{\star}$ is an $L$-dimensional system with single output, $\rmsK {\bm K}^{L}_{\star}$ is $L$-step observable. Thus, Claim a) in Lemma \ref{lem2} follows. Therefore, there exists $K \in \mathbb R^{n_u \times n_z}$ in \eqref{IOH_law} being equivalent to $\rmsK {\bm K}^{L}_{\star}$. Thus, $J(K_{\star}) \leq \rmsJ J(\rmsK {\bm K}^{L}_{\star})$. As $\rmsK {\bm K}^{L}_{\star}$ is globally optimal to \eqref{defJ_orig} in the form of \eqref{dyn_K}, there is no $K$ such that $J(K) < \rmsJ J(\rmsK {\bm K}^{L}_{\star})$. Therefore, \eqref{opteq2} follows. This completes the claim. \qed

\section{Derivation of \eqref{Jerr}}\label{prf_Jerr}
For ${\bm \Sigma}^{\epsilon}$ in \eqref{def_sigma_relaxed}, define $p^{\epsilon} := [y^{\epsilon\top}Q^{\frac{1}{2}}, u^{\top}R^{\frac{1}{2}}]^{\top}$. Then, the closed-loop $({\bm \Sigma}^{\epsilon}, {\bm K}^{\epsilon})$ with $p^{\epsilon}$ can be described as 
\begin{align}\label{cls}
\hspace{-4mm} ({\bm \Sigma}^{\epsilon}, {\bm K}^{\epsilon}): 
    \begin{cases}\hspace{-1pt}
    \hspace{-2mm}&h^{\epsilon}(t+1) = \Theta_K h^{\epsilon}(t) + \Pi_d d(t) + \delta (t)\\
    \hspace{-2mm}&p^{\epsilon}(t) = \Omega_K h^{\epsilon}(t) + \Xi_d d(t),
\end{cases}
\end{align}
where $\Omega_K$ is defined in \eqref{Jerr} and $\Xi_d := [\Upsilon^{\top}Q^{\frac{1}{2}}, 0]^{\top}$. Since $y^{\epsilon\top}(t)Qy^{\epsilon}(t) + u^{\top}(t)Ru(t) = \|p^{\epsilon}(t)\|^2$, from the $H_2$-optimal control theory\cite{chen2012optimal}, $J^{\epsilon}$ subject to \eqref{cls} can be described as 
\begin{align}
    J^{\epsilon}(K) &= \|\Omega_K ({\sf z}I - \Theta_K)^{-1} [\Pi_d V^{\frac{1}{2}}_d, \sqrt{\epsilon}I] + [\Xi_d V^{\frac{1}{2}}_d,0] \|_{H_2}^2 \nonumber \\
    & = J(K) + {\epsilon} \|\Omega_K ({\sf z}I - \Theta_K)^{-1}\|_{H_2}^2,  
\end{align}
where we used the fact $\|[G_1,G_2]\|_{H_2}^2 = \|G_1\|_{H_2}^2 + \|G_2\|_{H_2}^2$. Therefore, \eqref{Jerr} follows. \qed

\section{Proof of Lemma \ref{lem_nablaJ}}\label{prf_lem_nablaJ}


According to $H_2$-optimal control theory \cite{soderstrom2002discrete}, $J^{\epsilon}$ subject to \eqref{cls} can be written as
\begin{equation}\label{Je_initial}
J^{\epsilon}(K) = \tr(\Phi_KV_h^{\epsilon}) + {\rm c}
\end{equation}
where $\Phi_K \geq 0$ is defined in \eqref{lyap}, ${\rm c} := \tr(\Xi_dV_d\Xi_d^{\top})$ is a constant independent of $K$, and $V_h^{\epsilon} := \Pi_d V_d \Pi_d^{\top}+{\epsilon}I$. Let $K' := K + \Delta K$, and $\Phi_{K'} \geq 0$ be the solution to 
\begin{align}
    \Theta_{K'}^{\top} \Phi_{K'} \Theta_{K'} - \Phi_{K'} + \Psi^{\top}Q\Psi  + \Gamma^{\top}{K'}^{\top}R{K'}\Gamma = 0. \label{lyap_prime}
\end{align}
Denote $\Phi_{K'} =: \Phi_K + \Delta \Phi$. By subtracting LHS of \eqref{lyap_prime} from LHS of \eqref{lyap}, we have
\begin{align}
    \Theta_K^{\top}\Delta\Phi\Theta_K \h -\h \Delta \Phi \h+ \h {\rm sym}(\Gamma^{\top}\Delta K^{\top}W_K) \h +\h \mathcal O(\Delta^2) = 0, \label{delta_lyap}
\end{align}
where $\mathcal O(\Delta^2)$ is the higher order term with respect to $\Delta K$ and $\Delta \Phi$. This higher-order term is negligible in the limit $\|\Delta K\| \to 0$, as we are interested in the first-order derivative. 
In the sequel, we ignore this term because it is negligible. Multiplying \eqref{delta_lyap} by $Y_K$ from the left and taking the trace on both sides, we have
\begin{align}
    \tr\left(Y_K(\Theta_K^{\top}\Delta\Phi\Theta_K - \Delta \Phi)\right) \h +\h \tr\left(2Y_KW_K^{\top}\Delta K\Gamma\right) = 0, \label{delta_lyap2}  
\end{align}
where we used the fact that $\tr(YX) = \tr(X^{\top}Y)=\tr(YX^{\top})$ holds for any $X$ and a symmetric matrix $Y$. Note that the first term in LHS of \eqref{delta_lyap2} can be written as 
\begin{align}
    \tr\left((\Theta_KY_K\Theta_K^{\top} - Y_K)\Delta\Phi\right) = -\tr\left(\Delta\Phi V_h^{\epsilon}\right), 
\end{align}
where we used \eqref{EKH2}. Since $\tr\left(\Delta\Phi V_h^{\epsilon}\right) = J^{\epsilon}(K')-J^{\epsilon}(K)$, \eqref{delta_lyap2} can be written as 
\begin{align}
    J^{\epsilon}(K')-J^{\epsilon}(K) = \tr\left(
    (2W_KY_K\Gamma)^{\top}\Delta K
    \right).
\end{align}
Therefore, the claim follows. \qed

\section{Proof of Lemma \ref{lem_stab}}\label{prf_lem_stab}
First, we show the sufficiency. Since $\Theta_K$ is Schur stable, the Lyapunov equation \eqref{lyap} admits a unique solution $\Phi_K \geq 0$. Thus, it follows from \eqref{Je_initial} that $J^{\epsilon}(K) < \infty$.

Next, we show the necessity. To this end, we consider a control law 
\begin{align}
    {\bm K}^{\epsilon}_{u}: \quad  u(t) = Kz(t) + \delta_u(t), \quad \delta_u(t) \sim \mathcal N(0,\epsilon I), 
\end{align}
and the closed-loop $({\bm \Sigma}, {\bm K}^{\epsilon}_{u})$ described as 
\begin{align}\label{cls_ue}
\hspace{-4mm} ({\bm \Sigma}, {\bm K}_u^{\epsilon}): 
    \begin{cases}\hspace{-1pt}
    \hspace{-2mm}&h(t+1) = \Theta_K h(t) + \Pi_d d(t) + \Pi_u \delta_u (t)\\
    \hspace{-2mm}&p(t) = \Omega_K h(t) + \Xi_d d(t) + \Xi_u \delta_u(t)\\
    \hspace{-2mm}&z(t) = \Gamma h(t)
\end{cases}
\end{align}
where $\Xi_u := [0, R^{\frac{1}{2}}]^{\top}$, $\Omega_K$ and $\Xi_d$ are defined in \eqref{cls}. Denote $J$ in \eqref{cost} subject to this closed-loop system by $J_u^{\epsilon}$, i.e., 
\begin{align}\label{Jue}
    J_u^{\epsilon}(K) := \lim_{T\to\infty} \frac{1}{T}\mathbb E\left[\sum_{t=L}^{T}\|p(t)\|^2\right] 
\end{align}
where $p$ follows \eqref{cls_ue}. In the sequel, we will show the following two claims: 
\begin{enumerate}
    \item Given $K$ such that $J^{\epsilon}(K) < \infty$, it follows that $J_u^{\epsilon}(K) < \infty$. 
    \item If $J_u^{\epsilon}(K) < \infty$, then $K \in \mathbb K$.
\end{enumerate}
We first show Claim 1). Since $d(t)$ obeys an i.i.d. process, 
\begin{align}
    J^{\epsilon} = \lim_{T\to\infty}\mathbb E[\sum_{t=L}^{T}\|\Omega_Kh^{\epsilon}(t)\|^2/T] + {\rm c} \label{opm}
\end{align}
where ${\rm c} := \tr(\Xi_dV_d\Xi_d^{\top})$. By computing the first term in the RHS of \eqref{opm}, we have 
\begin{align}
    J^{\epsilon} < \infty \Rightarrow \tr\left(\lim_{T\to\infty}\left(\frac{1}{T}\sum_{k=1}^{T}\mathcal O_k^{\top}\mathcal O_k\right)V_h^{\epsilon}\right) < \infty, \label{Jeupper}
\end{align}
where $V_h^{\epsilon}$ is defined in \eqref{Je_initial} and $\mathcal O_k := \mathcal O_k(\Theta_K, \Omega_K)$. On the other hand, $J_u^{\epsilon}$ in \eqref{Jue} can be written as 
\begin{align}
    J_u^{\epsilon} &= \tr\left(\lim_{T\to\infty}\left(\frac{1}{T}\sum_{k=1}^{T}\mathcal O_k^{\top}\mathcal O_k\right)(\Pi_dV_d\Pi_d^{\top} + \epsilon \Pi_u\Pi_u^{\top})\right) \nonumber \\
    & + \lim_{T\to\infty}\frac{1}{T}\tr\left(\mathcal O_{T+1}\mathbb E[h(L)h(L)^{\top}]\mathcal O_{T+1}^{\top}\right) + \hat{\rm c}, \label{Jue_alter}
\end{align}
where $\hat{\rm c} := {\rm c} + \epsilon \tr(R)$. 
Note that $V_h^{\epsilon} \geq \Pi_dV_d\Pi_d^{\top} + \epsilon \Pi_u\Pi_u^{\top}$. Thus, \eqref{Jeupper} yields that the first term in RHS of \eqref{Jue_alter} is bounded. Furthermore, since $V_h^{\epsilon} > 0$,  \eqref{Jeupper} yields that $\lim_{T \to\infty}\mathcal O_{T}\mathcal O_{T}^{\top}/T = 0$. Therefore, \eqref{Jeupper} yields that $J_{u}^{\epsilon} < \infty$, which shows Claim 1). 

Next, we show Claim 2). To this end, we introduce the following lemma. 

\begin{lem}\label{J_lemma}
Consider $L \in \mathbb N$, $L$-step observable $\rms{\bm \Sigma}$ in \eqref{hat Sigma}, and $J_u^{\epsilon}$ in \eqref{Jue}. Suppose $J_u^{\epsilon}(K) < \infty$. Then
\begin{align}
    J_u^{\epsilon}(K) = \lim_{T \rightarrow \infty} \frac{1}{T} \mathbb E
    \left[\sum_{t=L}^{T}z^{\top}(t)S z(t) / L\right] + \epsilon\tr(R), 
\end{align}
where
 \begin{align}\label{defGss}
 S := {\rm diag}\left(I_L \otimes R, I_L \otimes Q\right) > 0  
 \end{align}
and $\otimes$ denotes the Kronecker product.
\end{lem}
\begin{proof}
From the definition of $p(t)$, we have $\mathbb E[\|p(t)\|^2] = \mathbb E[y^{\top}(t)Qy(t)+u^{\top}(t)Ru(t)] + \epsilon\tr(R)$. Therefore, $J_u^{\epsilon}(K) = J_u(K) + \epsilon\tr(R)$ with 
\begin{align}
    J_u(K) := \lim_{T \rightarrow \infty} \frac{1}{T}\mathbb E\left[\sum_{t=L}^{T}r(t)\right], \quad  r(t) := y^{\top}(t)Qy(t) + u^{\top}(t)Ru(t).
\end{align}
Since $r(t) \geq 0$ and $J_u^{\epsilon}(K)$ is bounded, by letting $T = \tau L$, we have $J_u(K) = \lim_{\tau \rightarrow \infty} \frac{1}{\tau L}\mathbb E\left[\sum_{t=L}^{\tau L - j-1}r(t)\right]$ for $j \in \{1, \ldots, L\}$. Note that
 \begin{align}
     \sum_{t=(k-1)L}^{kL-1} r(t-j)  = z^{\top}(kL-j)S z(kL-j) \label{J_alter}
 \end{align}
  holds for $k = 2,3,\cdots$, and $j = 1, \ldots, L$. 
  Therefore, we have
  \begin{align}
      \sum_{t=L-j}^{L-1} r(t) + \sum_{t=L}^{\tau L- j - 1} r(t) = \sum_{k=2}^{\tau} z^{\top}(kL-j)Sz(kL-j)
  \end{align}
  for $j \in \{1,\ldots,L\}$. Furthermore, $y(t)$ and $u(t)$ are bounded for any $t = 0, \ldots, L-1$. Therefore, we have 
\begin{align}
    \lim_{\tau \rightarrow \infty}\frac{1}{\tau L}\mathbb E\left[\sum_{t=L}^{\tau L- j - 1} r(t)\right] &= \lim_{\tau \rightarrow \infty} \frac{1}{\tau L} \mathbb E
    \left[\sum_{k=2}^{\tau}z^{\top}(kL-j)S z(kL-j)\right]
\end{align}
for $j \in \{1,\ldots,L\}$. By summing up this equality for $j = 1, \ldots, L$, we obtain
$L J_u(K) = \lim_{T \rightarrow \infty} \frac{1}{T} \mathbb E \left[\sum_{t=L}^{T}z^{\top}(t)S z(t)\right]. $
This completes the proof. 
\end{proof}
From this lemma, $J_u^{\epsilon} < \infty$ yields $\lim_{T \rightarrow \infty} \mathbb E[\|z(T)\|^2/T] = 0$. Furthermore, since $h = [z^{\top}, e^{\top}]^{\top}$, the dynamics of $e(t)$ in \eqref{cls_ue} can be written as $e(t+1) = \Theta_{22}e(t) + \Pi_{de}d(t)$, where $\Theta_{22} := {\rm diag}(J_{n_w},J_{n_v}) \in \mathbb R^{n_e \times n_e}$ and $\Pi_{de} := {\rm diag}(E_{n_w}, E_{n_v})\in \mathbb R^{n_e \times (n_w + n_v)}$ where $J_{n}$ and $E_n$ are defined in Lemma \ref{sys_lem}. Note that this dynamics is independent from $K$, and $\Theta_{22}$ is Schur stable. 
Thus, $\lim_{T \rightarrow \infty} \mathbb E[\|e(T)\|^2/T] = 0$ follows. Therefore, we have 
\begin{align}\label{hstable}
    \lim_{T \rightarrow \infty} \mathbb E[\|h(T)\|^2/T] = 0. 
\end{align}
This does not immediately yield that $\Theta_K$ is Schur stable because $h(t)$ is excited only by $d(t)$ and $\delta_u(t)$. In other words, for showing the stability of $\Theta_K$ from \eqref{hstable}, it suffices to show that the unreachable subspace of \eqref{cls_ue} is stable, i.e., the pair $(\Theta_K, [\Pi_dV_d^{\frac{1}{2}}, \Pi_u])$ is stabilizable. Since the reachability is invariant under the (partial) state-feedback control law\footnote{
Let the left-hand side (resp. right-hand side) of \eqref{image} be denoted by $\mathscr{R}_K$ (resp. $\mathscr{R}$). To show $\mathscr{R}_K = \mathscr{R}$, we prove both inclusions. $(\Rightarrow)$ Suppose $\bar{h} \in \mathscr{R}$. Then there exists a pair of input sequences $\{u(t), d(t)\}_{t=L}^T$ such that the open-loop system ${\bm \Sigma}$ satisfies $h(T) = \bar{h}$. Now consider the closed-loop system \eqref{cls_ue} with $\delta_u(t) = u(t) - K\Gamma h(t)$. By applying $\{\delta_u(t), d(t)\}_{t=L}^T$ to the closed-loop system, the same state trajectory is reproduced, implying $\bar{h} \in \mathscr{R}_K$. $(\Leftarrow)$ Conversely, for any $\bar{h} \in \mathscr{R}_K$, a similar construction shows that the open-loop system can reproduce the same trajectory of the closed-loop system. Thus, $\bar{h} \in \mathscr{R}$. Therefore, \eqref{image} holds.
}, we have
\begin{align}
    \im \mathcal R_{n_h}(\Theta_K, [\Pi_dV_d^{\frac{1}{2}}, \Pi_u]) = \im \mathcal R_{n_h}(\Theta, [\Pi_dV_d^{\frac{1}{2}}, \Pi_u]). \label{image}
\end{align}
Therefore, the stabilizability of $(\Theta_K, [\Pi_dV_d^{\frac{1}{2}}, \Pi_u])$ is equivalent to that $(\Theta, [\Pi_dV_d^{\frac{1}{2}}, \Pi_u])$ is stabilizable, which is satisfied because $\mathbb K \not=\emptyset$ from Assumption \ref{ass_noemp}. Therefore, \eqref{hstable} yields that $\Theta_K$ is Schur stable. This completes the proof. 
\qed

\begin{rem}\label{rem_Qdefinite}
Since $Q \geq 0$ yields $S \geq 0$, the condition $J_u^{\epsilon}(K) < \infty$ does not necessarily imply $\lim_{T \rightarrow \infty} \mathbb{E}[\|z(T)\|^2/T] = 0$, and thus the necessity part of Lemma 5 cannot be theoretically guaranteed. To avoid this issue without assuming the detectability of $(\Theta_K, \Psi)$, we assume $Q > 0$. 
\end{rem}

\section{Proof of Lemma \ref{lem_coer}} \label{prf_lem_coer}
\subsection{$J^{\epsilon}(K) \to \infty$ as $K \to \overline{K}\in\partial\mathbb K$}
We first show $\lim_{i \to \infty} J^{\epsilon}(K_i) = \infty$ for any sequence $K_i\to\overline K\in\partial\mathbb K$. From \eqref{Je_initial} and the fact that $\sigma_{\rm min}(V_h^{\epsilon}) \geq \epsilon >0$, we have $J^{\epsilon}(K) \geq \epsilon \tr(\Phi_K)$ where $\Phi_K$ is defined in \eqref{lyap}. Thus, it suffices to show that $\tr(\Phi_{K_i}) \to \infty$ as $K_i \to \overline{K}$ for any $\overline{K} \in \partial \mathbb{K}$. To this end, for each $s \in \{1, \dots, w\}$ we define the eigenprojection
\begin{align}
    P^{(s)}_K := \frac{1}{2\pi {\mathrm i}}\oint_{\mathcal C^{(s)}} (\zeta I-\Theta_K)^{-1}\,d\zeta, 
\end{align}
where ${\mathrm i}$ is the imaginary unit and $\mathcal C^{(s)}$ is the curve defined in Assumption \ref{ass_NOBIF}. Note that $P^{(s)}_K$ is a projection matrix satisfying $\rank P^{(s)}_K = m^{(s)}$. In this proof, we use the following notation: 
For $a_{ij}\in\mathbb R$ and $i,j \in \{0,\dots,m-1\}$, let $[a_{ij}]_{ij} \in \mathbb R^{m \times m}$ denote the matrix whose $(i+1,j+1)$-th entry is $a_{ij}$. For $A \in \mathbb R^{m \times m}$ and $k \in \{0, \dots, m-1\}$, ${\rm diag}_k(A) \in \mathbb R$ denotes its $(k+1)$-th diagonal entry. Throughout this proof, assume $|\lambda_K^{(s)}|\not=0$ for any $K \in \mathcal B_{\eta^{(s)}}(\overline{K})$ without loss of generality. Under these settings, we will prove $\tr(\Phi_{K_i}) \to \infty$ by the following steps. 

\begin{itemize}
    \item[(a)] On $\mathcal B_{\eta^{(s)}}(\overline{K})$, we show that the maps $K \mapsto P^{(s)}_K, \lambda^{(s)}_K, {\sf P}^{(s)}_K := P_K^{(s)}P_K^{(s)\dagger}$ and $N_K^{(s)} := (\Theta_K - \lambda_K^{(s)}I)P^{(s)}_K$ are continuous. Thus, so is 
    \begin{align}\label{def_Aks}
        A_k^{(s)}(K) := \Omega_K (N^{(s)}_K)^{k}{\sf P}_K^{(s)}, \quad \forall k \in \{0,\dots,m^{(s)}-1\}. 
    \end{align}

    \item[(b)] For $K \in \mathcal B^{\rm stab}_{\eta^{(s)}}(\overline{K}) := \mathcal B_{\eta^{(s)}}(\overline{K}) \cap \mathbb K$, we show 
    \begin{align}
    \hspace{-8mm}\tr(\Phi_K) &\geq \sum_{t=0}^{\infty}\left\|\sum_{k=0}^t \binom{t}{k}(\lambda^{(s)}_K)^{t-k}A_k^{(s)}(K)\right\|_F^2 \label{lb01} \\
    &\hspace{-8mm} \geq \frac{1}{\alpha_k(|\lambda_K^{(s)}|)} \|A_k^{(s)}(K)\|_F^2,\quad  k\in\{0,\ldots,m^{(s)}-1\} \label{lb02}
\end{align}
where $S(r) := \left[\sum_{t=0}^{\infty}r^{2t - (i+j)} \binom{t}{i} \binom{t}{j}\right]_{ij} \in \mathbb R^{m^{(s)} \times m^{(s)}}$ is positive definite for any $r \in (0,1)$ and $\alpha_k(r) := {\rm diag}_k(S(r)^{-1})$. 

\item[(c)] We prove $\lim_{r \to 1} 1/\alpha_k(r) = \infty$ for any $k \in \{0, \ldots, m^{(s)}-1\}$. 

\item[(d)] We show the existence of $s_{\star} \in \{1,\dots,w\}$ and $k_{\star} \in \{0, \ldots, m^{(s_{\star})}-1\}$ such that $\|A_{k_{\star}}^{(s_{\star})}(\overline{K})\|_F =: d >0$. 

\item[(e)] We prove $\tr(\Phi_{K_i}) \to \infty$. 
\end{itemize}
Note that in Steps (a)-(d), the index $s \in \{1, \dots, w\}$ is arbitrary but fixed. Thus, in the sequel we omit the superscript $^{(s)}$ when no confusion occurs. Moreover, for simplifying the notation, we denote $\mathcal B_\eta(\overline K)$ and $\mathcal B^{\rm stab}_\eta(\overline K)$  as $\mathcal B$ and $\mathcal B^{\rm stab}$, respectively.  

 {\bf [Step (a)]:}~ First, we show the continuity of $P_K$. Let $\mathcal R_K(\zeta):=(\zeta I-\Theta_K)^{-1}$. 
By Assumption~\ref{ass_NOBIF}, $\mathcal C$ lies in the resolvent set of $\Theta_K$ for all $K \in \mathcal B$. Since $\mathcal R_K^{-1}(\zeta)(\mathcal R_K(\zeta) - \mathcal R_{K'}(\zeta))\mathcal R_{K'}^{-1}(\zeta) = \Theta_K-\Theta_{K'}$, we obtain $P_K-P_{K'} =\frac{1}{2\pi {\mathrm i}}\oint_{\mathcal C}\mathcal R_K(\zeta)\,(\Theta_K-\Theta_{K'})\,\mathcal R_{K'}(\zeta)\,d\zeta$ for $K,K'\in\mathcal B$. 
Parameterize $\mathcal C$ by a $C^1$ map $\tau:[0,1]\to\mathbb C$, and write its arc length as
$a:=\int_0^1|d\tau/dt|dt$. For a matrix-valued function $M(\zeta)$, it follows that $\|\oint_{\mathcal C}M(\zeta)\,d\zeta\|
=\|\int_0^1 M(\tau(t))\,\tau'(t)\,dt\|
\leq a\sup_{\zeta\in\mathcal C}\|M(\zeta)\|$. Hence
\[
\|P_K-P_{K'}\|
\le \frac{a}{2\pi}\,
\sup_{\zeta\in\mathcal C}\|\mathcal R_K(\zeta)\|\;
\|\Theta_K-\Theta_{K'}\|\;
\sup_{\zeta\in\mathcal C}\|\mathcal R_{K'}(\zeta)\|.
\]
Since $(\zeta,K)\mapsto \mathcal R_K(\zeta)$ is continuous on the compact set
$\mathcal C\times\overline{\mathcal B_\eta(\overline K)}$, there exists $b:=\sup_{\zeta\in\mathcal C,\;K\in\mathcal B}\|\mathcal R_K(\zeta)\|<\infty$. 
Therefore,
\[
\|P_K-P_{K'}\| \le \rho \|K - K'\|, \quad \rho := \frac{ab^2}{2\pi}\|\Gamma\|~\|\Pi\|, 
\]
thus $K\mapsto P_K$ is continuous on $\mathcal B$. 
Next, we prove the continuity of ${\sf P}_K:=P_K P_K^{\dagger}$ on $\mathcal B$.
By Assumption~\ref{ass_NOBIF}, $\operatorname{rank} P_K$ is constant on $\mathcal B$, and we have already shown that $K\mapsto P_K$ is continuous.
Hence, $K\mapsto P_K^{\dagger}$ is continuous \cite{xu2020perturbation}, and therefore so is $K\mapsto{\sf P}_K$. Moreover, since $\mathcal C$ encloses a single eigenvalue with fixed multiplicity $m$, $\lambda_K$ is continuous. Consequently $N_K$ and $A_k(K)$ in~\eqref{def_Aks} are continuous on $\mathcal B$. 

 {\bf [Step (b)]:}~ Since $K \in \mathcal B^{\rm stab}$, $|\lambda_K| \in (0,1)$ holds. Note that $P_K^2 = P_K$, $N_KP_K = P_KN_K$, and $P_K {\sf P}_K = {\sf P}_K$. Therefore,  
 $\Theta_K^t {\sf P}_K = \sum_{k=0}^t \binom{t}{k}\lambda_K^{t-k}N_K^{k}{\sf P}_K$. Thus, ${\rm tr}(\Phi_K) \geq \tr \left({\sf P}_K^{*}\Phi_K {\sf P}_K\right) = \sum_{t=0}^{\infty}\|\Omega_K\Theta_K^t {\sf P}_K\|_F^2 = \sum_{t=0}^{\infty} \|\sum_{k=0}^{t}\binom{t}{k} \lambda_K^{t-k} A_k(K)\|_F^2$, which yields \eqref{lb01}. 
 
 Next, we show \eqref{lb02}. Let $S_{ij}(r) \in \mathbb R$ be the $(i+1, j+1)$-th entry of $S(r)$, $\lambda_K =: |\lambda_K|e^{{\mathrm i}\theta_K}$ and $\mathcal A_k(K) := e^{-{\mathrm i}k\theta_K}A_k(K)$. Then
 \begin{align}
     \HH {\tr }(\Phi_K) \geq \sum_{j=0}^{m-1}\sum_{i=0}^{m-1} S_{ij}(|\lambda_K|)\tr(\mathcal A_i^*\mathcal A_j) = \tr (\mathcal A^*(S(|\lambda_K|)\otimes I)\mathcal A) \label{tttt}
 \end{align}
 where $\mathcal A := [\mathcal A_0^{\top}, \dots, \mathcal A_{m-1}^{\top}]^{\top}$. 
 
 As a first step to prove \eqref{lb02}, we show $S(r) > 0$ for any $r \in (0,1)$.  
 Let
 \[
 s_t(r) := \left[\binom{t}{0}r^t, \ldots, \binom{t}{m-1}r^{t-(m-1)}\right]^{\top} \in \mathbb R^m. 
 \]
 Then, $S(r) = \sum_{t=0}^{\infty}s_t(r)s^{\top}_t(r)$. Let $W(r) := \left[s_0(r), \dots, s_{m-1}(r)\right] \in \mathbb R^{m \times m}$. Then, $\sum_{t=0}^{m-1}s_t(r)s_t^{\top}(r) = W(r)W^{\top}(r)$. Since $\binom{t}{i}=0$ for $i>t$, $W(r)$ is upper triangular matrix whose every diagonal entry is $1$ regardless $r$. Therefore, ${\rm det}(W(r)W(r)^{\top}) = {\rm det}^2(W(r))=1$. Hence, $S(r) \geq W(r)W(r)^{\top} > 0$. 

 Now we prove \eqref{lb02}. For simplifying the notation, we denote $S(|\lambda_K|)$ as $S$. 
Let $\Pi_k := [e_k, e_0, \dots, e_{k-1}, e_{k+1}, \dots, e_{m-1}]\in\mathbb{R}^{m\times m}$ for $k \in \{0, \dots,m-1\}$ where $e_k$ is the $(k+1)$-th column vector of $I_m$. Let $\widetilde S := \Pi_k^\top S\Pi_k$ and $\widetilde{\mathcal A} := (\Pi_k^{\top}\otimes I) {\mathcal A} = [{\mathcal A}_k^{\top}, \overline{{\mathcal A}}^{\top}]^{\top}$ where $\overline{X}\in\mathbb{C}^{(m-1)n\times n}$ stacks all remaining blocks. Note that $\widetilde{S} > 0$. Partition $\widetilde{S}$ as $\widetilde{S} =: \left[\begin{array}{cc} \widetilde{S}_{11} & \widetilde{S}_{21}^{\top} \\ \widetilde{S}_{21} & \widetilde{S}_{22}\end{array}\right]$ where $\widetilde{S}_{11} \in \mathbb R$. Then, 
 \begin{align}
     &\tr \left({\mathcal A}^*(S \otimes I){\mathcal A}\right) = \tr \left(\widetilde{{\mathcal A}}^*(\Pi_k^{\top} S \Pi_k \otimes I)\widetilde{{\mathcal A}}\right) = \tr \left(\widetilde{{\mathcal A}}^*(\widetilde{S} \otimes I)\widetilde{{\mathcal A}}\right) \nonumber \\
     &=\widetilde{S}_{11}\|{\mathcal A}_k\|_F^2 - \tr({\mathcal A}_k^*\widetilde{S}_{21}^{\top}\widetilde{S}_{22}^{-1}\widetilde{S}_{21}{\mathcal A}_k)  \nonumber \\
     &+ \tr\left((\overline{{\mathcal A}}+(\widetilde{S}_{22}^{-1}\widetilde{S}_{21}\otimes I){\mathcal A}_k)^*(\widetilde{S}_{22}\otimes I)(\overline{{\mathcal A}}+(\widetilde{S}_{22}^{-1}\widetilde{S}_{21}\otimes I){\mathcal A}_k)\right) \nonumber \\
     & \geq (\widetilde{S}_{11} - \widetilde{S}_{21}^{\top}\widetilde{S}_{22}^{-1}\widetilde{S}_{21})\|{\mathcal A}_k\|_F^2 = \frac{1}{{\rm diag}_k(S^{-1})}\|{\mathcal A}_k\|_F^2  \label{tttt2}
 \end{align}
 for $k \in \{0, \dots, m-1\}$, where the last equality follows from that $(\widetilde{S}_{11} - \widetilde{S}_{21}^{\top}\widetilde{S}_{22}^{-1}\widetilde{S}_{21}) = {\rm diag}_0(\widetilde{S}^{-1}) = {\rm diag}_k({S}^{-1})$. Since $\|\mathcal A_k\|_F^2 = \|A_k\|_F^2$, we have \eqref{lb02}. 

{\bf [Step (c)]:}~ Since $S(r) > 0$ for any $r \in (0,1)$, we have $\alpha_k(r) > 0$ for any $k \in \{0, \ldots, m-1\}$ and $r \in (0,1)$. Let $D(r) := {\rm diag}((1-r^2)^{\frac{1}{2}}, (1-r^2)^{\frac{3}{2}}, \ldots, (1-r^2)^{m-\frac{1}{2}}) \in \mathbb R^{m \times m}$ and $M(r) := D(r)S(r)D(r)$. If there exist constant $c>0$ and $\underline{\lambda}>0$ such that 
\begin{align}
 M(r) \geq c I_m, ~^\forall r\in(\underline{\lambda}, 1), \label{toshowww}
\end{align}
then, we have $S(r) = D(r)^{-1}M(r)D(r)^{-1} \geq c D(r)^{-2}$. Thus, $S(r)^{-1} \leq c^{-1} {\rm diag}((1-r^2), \ldots, (1-r^2)^{2m-1})$. Multiplying $e_k^{\top}$ and $e_k$ from the left and right yields ${\rm diag}_k(S(r)^{-1}) \leq c^{-1}(1-r^2)^{2k+1}$ for any $k \in \{0,\ldots,m-1\}$. Therefore
\[
 \frac{1}{\alpha_k(r)} \geq \frac{c}{(1-r^2)^{2k+1}}, \quad ^\forall r \in (\underline{\lambda},1)
\]
which yields $\lim_{r \to 1} 1/\alpha_k(r) = \infty$ for any $k \in \{0,\ldots,m-1\}$. Hence, in the remainder of Step (c), it suffices to show \eqref{toshowww}. To this end, we introduce the following lemma: 
\begin{lem}
Let $t\in\mathbb{N}$ and $k\in\{0,\dots,t\}$. Define $(t)_k:=t(t-1)\cdots(t-k+1)$ with $(t)_0:=1$. Then for all $i,j\in\{0,\dots,t\}$,
\begin{align}
(t)_i(t)_j &= \sum_{k=0}^{\min(i,j)} \binom{i}{k}\binom{j}{k}k!(t)_{i+j-k}, \label{falling-factorial-id} \\
\sum_{t=0}^{\infty} (t)_k r^t &= \frac{k!\,r^k}{(1-r)^{k+1}}, \qquad r\in(0,1).
\label{falling-factorial-genfun}
\end{align}
\end{lem}
\begin{proof}
We first prove \eqref{falling-factorial-id}. Since $\sum_{i=0}^t \binom{t}{i}x^i = (1+x)^t$, we have
\begin{align}
&\left(\sum_{i=0}^t \binom{t}{i}x^i \right)\left(\sum_{j=0}^t \binom{t}{j}y^j \right)
= (1+x+y+xy)^t \nonumber \\ 
&= \sum_{p=0}^t\binom{t}{p}\bigl(y + x(1+y)\bigr)^p = \sum_{p=0}^t \binom{t}{p} \sum_{q=0}^p \binom{p}{q}y^{p-q}(x+xy)^q. \label{stepC02}
\end{align}
To extract the coefficient of $x^iy^j$ in \eqref{stepC02}, note that $x$ appears only from the factor $(x+xy)^q$, hence $q=i$. Substituting $q=i$ yields $\sum_{p=i}^t \binom{t}{p} \binom{p}{i} y^{p-i} x^i(1+y)^i$, where the sum starts at $p=i$ because $\binom{t}{p}=0$ for $p<i$. Expanding $(1+y)^i$ and again comparing the coefficient of $y^j$ gives
\begin{align}
    \binom{t}{i}\binom{t}{j} =\sum_{p=i}^t \binom{t}{p} \binom{p}{i} \binom{i}{j-(p-i)}.\label{stepC03}
\end{align}
Since the last binomial vanishes unless $j-(p-i)\in[0,i]$, the valid range of $p$ is 
$p\in[\max(i,j),i+j]$. Let $k:=j-(p-i)$, whose range is $k\in[0,\min(i,j)]$. After a change of variables we obtain
\begin{align}
    \binom{t}{i}\binom{t}{j} &=\sum_{k=0}^{\min(i,j)} \binom{t}{i+j-k} \binom{i+j-k}{i} \binom{i}{k} \nonumber \\
    &= \sum_{k=0}^{\min(i,j)}  \binom{t}{i+j-k} \frac{(i+j-k)!}{(i-k)!(j-k)!k!}\label{stepC04}
\end{align}
Finally, using $\binom{t}{i}i!=(t)_i$, multiplying both sides by $i!j!$ gives 
\[
(t)_i(t)_j = \sum_{k=0}^{\min(i,j)} (t)_{i+j-k} \frac{i!j!}{(i-k)!(j-k)!k!} = \sum_{k=0}^{\min(i,j)}\binom{i}{k}\binom{j}{k}k!(t)_{i+j-k},
\]
which proves \eqref{falling-factorial-id}. For \eqref{falling-factorial-genfun}, consider $G(r):=\sum_{t=0}^\infty r^t=\frac{1}{1-r}$, which is analytic on $(0,1)$. Differentiating term-wise, we have
\[
\frac{k!}{(1-r)^{k+1}} =\frac{d^k}{d r^k}G(r) = \sum_{t=0}^{\infty} \frac{d^k}{dr^k}r^t = \sum_{t=0}^\infty (t)_k r^{t-k}. 
\]
Multiplying by $r^k$ yields the claim. 
\end{proof}

Let $M_{ij}(r) \in \mathbb R$ denotes the $(i+1,j+1)$-th entry of $M(r)$. By definition, 
\begin{align}
        M_{ij}(r) &= (1-r^2)^{i+\frac{1}{2}}\sum_{t=0}^{\infty}r^{2t-(i+j)}\binom{t}{i}\binom{t}{j}(1-r^2)^{j+\frac{1}{2}} \nonumber \\
        & =  r^{-(i+j)}(1-r^2)^{i+j+1}T_{ij}(r) \label{stepC044}
\end{align}
where 
\begin{align}
&T_{ij}(r) := \sum_{t=0}^{\infty}\binom{t}{i}\binom{t}{j} r^{2t} = \frac{1}{i!j!}\sum_{t=0}^{\infty} (t)_i(t)_jr^{2t} \nonumber \\
&= \frac{1}{i!j!}\sum_{k=0}^{\min(i,j)}\binom{t}{k}\binom{j}{k}k!\sum_{t=0}^{\infty} (t)_{i+j-k} r^{2t} \label{stepC05} \\
&= \frac{1}{i!j!}\sum_{k=0}^{\min(i,j)}\binom{t}{k}\binom{j}{k}k! \frac{(i+j-k)! r^{2(i+j-k)}}{(1-r^2)^{i+j-k+1}}, \label{stepC06}
\end{align} 
where \eqref{falling-factorial-id} and \eqref{falling-factorial-genfun} are used for deriving \eqref{stepC05} and \eqref{stepC06}, respectively. Multiplying $(1-r^2)^{i+j+1}$ in both side, we have $(1-r^2)^{i+j+1}T_{ij}(r) = \frac{1}{i!j!}\sum_{k=0}^{\min(i,j)}f(r; k, i,j)$ where 
\begin{align}
    f(r; k, i,j) := \binom{t}{k}\binom{j}{k}k!(i+j-k)! r^{2(i+j-k)}(1-r^2)^k. 
\end{align}
Note that $\lim_{r \to 1}f(r; k,i,j) = 0$ for any $k \geq 1$, and $f(r; 0,i,j) = (i+j)! r^{2(i+j)}$. Thus, $\lim_{r \to 1}(1-r^2)^{i+j+1}T_{ij}(r) = \binom{i+j}{i}$. Therefore, from \eqref{stepC044}
\begin{align}
\lim_{r \to 1}M(r) = \left[\binom{i+j}{i}\right]_{ij} =: \mathcal M \in \mathbb R^{m \times m}. 
\end{align}
Let $\mathcal L:= \left[\binom{i}{j}\right]_{ij} \in \mathbb R^{m \times m}$. Note that $\mathcal L \mathcal L^{\top} = \left[\sum_{k=0}^{\min(i,j)} \binom{i}{k}\binom{j}{k}\right]_{ij} = \left[\sum_{k=0}^{\min(i,j)} \binom{i}{i-k}\binom{j}{k}\right]_{ij} = \left[\binom{i+j}{i}\right]_{ij}$, where the last equality holds from Vandermonde's Identity. Since $\mathcal L$ is lower triangular matrix whose every diagonal element is $1$. Thus, \(\det\mathcal L=1\). Hence, $\det \mathcal M=\det^2 \mathcal L > 0$, which yields that $\mathcal M > 0$. Since $M(r)$ depends continuously on $r$ and the smallest singular value is a continuous function of the matrix entries, there exist constant $c>0$ and $\underline{\lambda}>0$ such that \eqref{toshowww}. 

{\bf [Step (d)]:}~ Assume to the contrary that $A_k^{(s)}(\overline K)=0$ for all $s\in\{1,\dots,w\}$ and $k \in \{0,\dots,m^{(s)}-1\}$. Then, we have
\begin{align}
\Omega_{\overline K}\Theta_{\overline K}^{\,t}{\sf P}^{(s)}_{\overline{K}}
=\sum_{k=0}^{m^{(s)}-1}\binom{t}{k}(\lambda^{(s)})^{t-k}\Omega_{\overline K}(N^{(s)})^{k}{\sf P}^{(s)}_{\overline{K}} =0, \quad \forall t \geq 0 \label{toshowStepd01}
\end{align}
Let ${\mathcal P}_{\overline{K}} := \sum_{s=1}^w {\sf P}^{(s)}_{\overline{K}}$. By \eqref{toshowStepd01}, $\Omega_{\overline K}\Theta_{\overline K}^{t}{\mathcal P}_{\overline{K}}=0$ for all $t$. Since
$
\Omega_{\overline K}\Theta_{\overline K}^{t}
=\Omega_{\overline K}\Theta_{\overline K}^{t}{\mathcal P}_{\overline{K}} + \Omega_{\overline K}\Theta_{\overline K}^{\,t}(I-{\mathcal P}_{\overline{K}})
$, we have 
\[
\|\Omega_{\overline K}\Theta_{\overline K}^{\,t}\|_F
=\|\Omega_{\overline K}\Theta_{\overline K}^{\,t}(I-{\mathcal P}_{\overline{K}})\|_F
\le \|\Omega_{\overline K}\|\,\|\Theta_{\overline K}^{\,t}(I-{\mathcal P}_{\overline{K}})\|_F.
\]
All eigenvalues of $\Theta_{\overline K}$ projected onto ${\mathcal P}_{\overline{K}}$ lie strictly inside the unit disk, so by the Jordan–form estimate there exist $C>0$, $\rho_\star\in(0,1)$, and an integer $p\ge0$ such that
\[
\|\Theta_{\overline K}^{t}(I-{\mathcal P}_{\overline{K}})\|_F \ \le\ C\,t^{p}\,\rho_\star^{t}, \quad t \geq 0. 
\]
From \eqref{Je_initial}, we have $J^{\epsilon}(K) \leq {\rm c} + \|V_h^{\epsilon}\| \tr(\Phi_K)$. Therefore, 
\begin{align}
J^{\epsilon}(\overline K) &\leq {\rm c} + \|V_h^{\epsilon}\|\sum_{t=0}^\infty \|\Omega_{\overline K}\Theta_{\overline K}^{\,t}\|_F^2 \nonumber \\
 &\leq {\rm c} +\|V_h^{\epsilon}\|\|\Omega_{\overline K}\|^2\,C^2\sum_{t=0}^\infty t^{2p}\rho_\star^{\,2t}\ <\ \infty.
\end{align}
This contradicts Lemma \ref{lem_stab} as $\overline{K} \notin \mathbb K$. Hence there must exist
$s_\star$ and $k_\star\in\{0,\dots,m^{(s_\star)}\!-1\}$ satisfying $\|A_{k_\star}^{(s_\star)}(\overline K)\|_F>0$.

{\bf [Step (e)]:}~Fix $\{s_\star,k_\star\}$ given by Step~(d). As shown in Step (a), $A_{k_\star}^{(s_\star)}(K)$ is continuous on $K \in \mathcal B_{\eta^{(s_{\star})}}(\overline{K})$. Thus, $\lim_{K \to \overline{K}} \|A_{k_\star}^{(s_\star)}(K)\|_F = d >0$ and $\lim_{K \to \overline{K}} |\lambda_K^{(s_{\star})}| = 1$. Let $\{K_i\} \subset \mathcal B_{\eta^{(s_\star)}}^{\rm stab}(\overline{K})$ with $K_i\to\overline K$. Therefore, using the results shown in Steps (b)-(c), we have 
\begin{align}
    \lim_{i \to \infty} \tr(\Phi_{K_i}) \geq \lim_{i \to \infty}\frac{1}{\alpha_{k_{\star}}(|\lambda_{K_i}^{(s_{\star})}|)} \|A_{k_{\star}}^{(s_{\star})}(K_i)\|_F^2 = \infty. \label{toshowe01}
\end{align}
Since \eqref{toshowe01} holds for any $\overline{K} \in \partial \mathbb K$, it follows that $J^\epsilon(K_i)\to\infty$ for any $\{K_i\}\subset\mathbb K$ with $K_i\to\overline K\in\partial\mathbb K$.

\subsection{$J^{\epsilon}(K) \to \infty$ as $\|K\| \to \infty$}
Note that  
   \begin{align}
       J^{\epsilon}(K) = {\rm tr}\left(
       (\Psi^{\top}Q\Psi  + \Gamma^{\top}K^{\top}RK\Gamma)Y_K
       \right) + {\rm c}
   \end{align}
   where $Y_K$ is defined in \eqref{EKH2}, and ${\rm c} \geq 0$ is a constant (independent of $K$) defined in \eqref{Je_initial}. Note that, by definition, $Y_K \geq \Pi_d V_d \Pi_d^{\top}+{\epsilon}I$. Thus, we have 
   $
    \sigma_{\rm min}(Y_K) \geq {\epsilon}.     
   $
   Thus, we have
   \begin{align} \label{Jlowerbound}
       J^{\epsilon}(K) \geq {\epsilon} \tr(\Gamma^{\top}K^{\top}RK\Gamma) = {\epsilon} \tr(RKK^{\top}), 
   \end{align}
   where $\Gamma \Gamma^{\top} = I$ is used for deriving the third equation. Using the formula: $\tr(AB) \geq \sigma_{\rm min}(B)\tr(A)$ holds for any $A \geq 0$ and $B > 0$, \eqref{Jlowerbound} yields 
   \begin{align}
       J^{\epsilon}(K) \geq \sigma_{\rm min}(R){\epsilon} \|K\|_F^2. \label{lowbound}
   \end{align}
   Note that $\|K\|_F \geq \|K\|$ for any matrix $K$. Now consider a sequence $\{K_i\} \subseteq \mathbb{K}$ such that $\|K_i\| \to \infty$. From  \eqref{lowbound}, we have 
\[
\lim_{i \to \infty}J^{\epsilon}(K_i) \geq  \sigma_{\rm min}(R)\epsilon \lim_{i \to \infty}\|K_i\|^2 \to \infty,
\]
which completes the proof.
\qed

\section{Proof of Lemma \ref{lem_hesian}} \label{prf_lem_hesian}
As described in Proof of Lemma \ref{lem_nablaJ}, $J^{\epsilon}$ in \eqref{cost_relax} subject to \eqref{cls} can be equivalently written as \eqref{Je_initial}. Therefore, similarly to deriving the final inequality in page 31 of \cite{bu2019lqr}, for any $K \in \mathbb K_c$ we have that
\begin{align}
   \hspace{0mm}\|\nabla^2 J^{\epsilon}(K)\|\leq 2 a_K \|Y_K\| &+ 4 \left\|\Theta_K Y_K^{\frac{1}{2}} \right\| \underset{\|K'\|_F=1}{\sup} \left\|\Gamma^{\top}K'^{\top}\Pi_u^{\top}X\right\| {\rm tr}(Y_K^{\frac{1}{2}}) \label{l8_pf2},
\end{align}
where $a_K := \|\Pi_u^{\top}\Phi_{K}\Pi_u + R \|$,  $X := (\partial \Phi_{K_{\alpha'}} / \partial \alpha')_{\alpha' =0}$, $K_{\alpha'} := K + \alpha'K'$, $\Phi_K$ and $Y_K$ are defined in \eqref{lyap}-\eqref{EKH2}. It suffices to show that the RHS of \eqref{l8_pf2} is bounded.

For the supremum term in \eqref{l8_pf2}, similarly to the last inequality in page 32 of \cite{bu2019lqr}, it follows that 
\begin{align} 
    X - \Theta_K^{\top}X \Theta_K \leq \Lambda_{K,K'}, \label{XX}
\end{align}
where $\Lambda_{K,K'} := \Phi_K - \Psi^{\top}Q\Psi + \Gamma^{\top}K'^{\top}\Pi_u^{\top}\Phi_K\Pi_uK'\Gamma+\Gamma^{\top}K'^{\top}RK'\Gamma\geq 0$. By using the facts $\|\Pi_u\| = 1$ and $\sup_{\|A\|_F \leq 1}\|AB\| \leq \|B\|$ for any $A$ and $B$, we have
\begin{align}
    \Lambda_{K,K'} \leq \Phi_K - \Psi^{\top}Q\Psi + (\|\Phi_K\|+\|R\|)I =: N_K. 
\end{align}
Note that $N_K \geq 0$.  Because $\Theta_K$ is Schur stable and $\mathbb K_{c}$ is a compact set, there exists $a_1$ such that 
\begin{align}
    \|X\| \leq \left\|\sum_{l=0}^{\infty}(\Theta_K^{\top})^lN_K \Theta_K\right\| \leq a_1 \label{def_a1}
\end{align}
holds for any $K \in \mathbb K_c$. Therefore, 
\begin{align}
\underset{\|K'\|_F=1}{\sup} \left\|\Gamma^{\top}K'^{\top}\Pi_u^{\top}X\right\| \leq a_1 \label{new_expa}.
\end{align} 
Similarly to \eqref{def_a1}, there exist $a_2$ such that 
\begin{align}
    \tr(Y_K) \leq a_2, \label{a2a3}
\end{align}
which yields $\|Y_K\| \leq a_2$. Using this relation, we show an upper bound of $\|\Theta_K Y_K^{1/2}\|$. Note that $\Theta_K Y_K\Theta_K^{\top} \leq Y_K$ follows from the definition of $Y_K$ and that $\Theta_K$ is Schur. Hence, 
\begin{align}\label{boundThKYKh}
   \left\|\Theta_K Y_K^{\frac{1}{2}}\right\| = \left(\|\Theta_K Y_K^{\frac{1}{2}}\|^2 \right)^{\frac{1}{2}} = \|\Theta_K Y_K\Theta_K^{\top}\|^{\frac{1}{2}} \leq \sqrt{a_2}. 
\end{align}
Furthermore, as in the proof of Proposition 7.7 in \cite{bu2019lqr}, we obtain 
\begin{align} \label{boundtrYKh}
{\rm tr}(Y_K^{\frac{1}{2}}) \leq \sqrt{2a_2}. 
\end{align} 
Furthermore, from \eqref{Je_initial} and the fact that $V_h^{\epsilon}$ in \eqref{Je_initial} satisfies ${\epsilon}I \leq V_h^{\epsilon}$, we have $\|\Phi_K\| \leq \tr(\Phi_K) \leq \tr(\Phi_KV_h^{\epsilon})/{\epsilon} = (J^{\epsilon}(K)-{\rm c})/{\epsilon} \leq c/{\epsilon}$. By substituting this, \eqref{new_expa}-\eqref{boundtrYKh} into \eqref{l8_pf2}, we obtain
\begin{align}\label{exact_q}
    \|\nabla^2 J^{\epsilon}(K)\|\leq q := 2(\|R\| + c{\epsilon}^{-1})a_2 + 4\sqrt{2}a_1a_2. 
\end{align}
This completes the claim.   
\qed

\section{Proof of Theorem \ref{linear_convergence}} \label{prf_linear_convergence}
We first show the following lemma: 
\begin{lem}\label{lem10}
    Given $K \in \mathbb K_c$ such that $K \not=\overline{K}$, define $K_{\alpha} := K - \alpha \nabla J^{\epsilon}(K)$. Let $\alpha_{\rm M}:= \max \alpha$ such that $K_{\alpha} \in \mathbb K_c$ for all $\alpha \in [0, \alpha_{\rm M}]$. Then, $\alpha_{\rm M} \geq 2/q$. 
\end{lem}
\begin{proof}
    The existence of $\alpha_{\rm M}$ follows from the compactness of $\mathbb K_c$. For the sake of contradiction, suppose $\alpha_{\rm M} < 2/q$. From the continuity of $J^{\epsilon}(K_{\alpha})$ with respect to $\alpha$, it follows that $J^{\epsilon}(K_{\alpha_{\rm M}}) = c$. Moreover, since $-\nabla J^{\epsilon}(K)$ is a descent direction of $J^{\epsilon}(K)$, we have $\alpha_{\rm M} > 0$. Therefore, since $J^{\epsilon}(K_{\alpha}) \leq c$ for any $\alpha \in [0, \alpha_{\rm M}]$, we have 
    \begin{align}
        J^{\epsilon}(K_{\alpha}) \h-\h J^{\epsilon}(K) \h&\leq\h \tr \H \left(\nabla J^{\epsilon}(K)(K_{\alpha}-K)^{\top}\right) \h+\h \frac{q}{2}\|K_{\alpha}-K\|_F^2 \nonumber \\
        &= -\frac{\alpha(2-q\alpha)}{2}\|\nabla J^{\epsilon}(K)\|_F^2 < 0, \label{prrrr}
    \end{align}
    where \eqref{q_def} is used for deriving the first inequality, and the relation $\alpha \leq \alpha_{\rm M} < 2/q$ is used for deriving the last inequality. Therefore, we have $J^{\epsilon}(K_{\alpha}) < J^{\epsilon}(K) \leq c$ for any $\alpha \in [0, \alpha_{\rm M}]$, which contradicts $J^{\epsilon}(K_{\alpha_{\rm M}}) = c$. This completes the proof. 
\end{proof}
We show \eqref{statio_Je}. Based on Lemma \ref{lem10}, it follows that 
\begin{align}
    K_i \in \mathbb K_c \Rightarrow K_{i+1} \in \mathbb K_c, \quad \forall \alpha \in (0, 2/q). 
\end{align}
Note here that this holds if $K_i = \overline{K}$ because then $K_{i+1} = K_i$. By letting $K \leftarrow K_i$ and $K_{\alpha} \leftarrow K_{i+1}$ in \eqref{prrrr}, we obtain
\begin{align}\label{interm}
    J^{\epsilon}(K_{i+1}) \leq J^{\epsilon}(K_i) - \alpha\left(1 - \frac{q\alpha}{2}\right) \|\nabla J^{\epsilon}(K_i)\|_F^2
\end{align}
Thus, as long as $\|\nabla J^{\epsilon}(K_i)\| \not=0$, $J^{\epsilon}(K_{i+1}) < J^{\epsilon}(K_i) \leq c$ if $\alpha \in (0, 2/q)$. From Lemma \ref{lem_stab}, this yields that $\Theta_{K_{i+1}}$ is Schur stable; thus $K_{i+1} \in \mathbb K_c$. Therefore, \eqref{interm} holds for any $i \geq 0$. Summing up \eqref{interm} yields 
\begin{align}
    \alpha\left(1 - \frac{q\alpha}{2}\right) \sum_{i=0}^N\|\nabla J^{\epsilon}(K_i)\|_F^2 \leq J^{\epsilon}(K_0) - J^{\epsilon}(K_{N+1}) < \infty. \nonumber
\end{align}
This yields \eqref{statio_Je}. Finally, we show \eqref{statio_Jori}. To this end, we introduce the following lemma. 
\begin{lem}\label{lem_stationary}
    Consider $J$ in \eqref{cost} and $J^{\epsilon}$ in \eqref{cost_relax}. Let $\overline{K} \in \mathbb K_c$ be a stationary point of $J^{\epsilon}$, i.e., $\nabla J^{\epsilon}(\overline{K}) = 0$. Then, there exists $\beta > 0$ satisfying 
    \begin{align}\label{eps_stationary}
        \|\nabla J(\overline{K})\|_F \leq \beta {\epsilon}.
    \end{align}
\end{lem}
\begin{proof}
By taking the derivative of \eqref{Jerr} with respect to $K$, we have $\nabla J^{\epsilon}(K) = \nabla J(K) + \nabla \gamma_K{\epsilon}$. By letting $K \leftarrow \overline{K}$, we have $\nabla J(\overline{K}) = -\nabla \gamma_{\overline{K}}{\epsilon}$. For any $K \in \mathbb K$, note here that $\gamma_K = \tr(\Phi_K)$ where $\Phi_K \geq 0$ is the solution to \eqref{lyap}. Thus, similarly to Lemma~\ref{lem_nablaJ}, $\nabla \gamma_K = 2W_K(\sum_{t=0}^{\infty}\Theta_K^t \Theta_K^{t\top})\Gamma^{\top}$ where $W_K$ is defined in \eqref{EKH1}. Since $\mathbb K_c$ is compact, there exists $\beta > 0$ such that $\|\nabla \gamma_{K}\|_F \leq \beta $ holds for any $K \in \mathbb K_c$. Therefore, the claim follows.  
\end{proof}
This lemma and \eqref{statio_Je} yield \eqref{statio_Jori}. This completes the proof. \qed

\bibliographystyle{elsarticle-num}
\bibliography{refs}

\end{document}